\documentclass[final,a4paper,leqno]{siamltex}

\usepackage[english]{babel}
\usepackage{amsmath,amssymb,amsfonts,graphicx}
\usepackage[ruled]{algorithm2e} % ubuntu package: texlive-science 
\usepackage{myshortcuts}
\usepackage{comment}
\usepackage{tikz}
\usetikzlibrary{external}
\usepackage[latin1]{inputenc} % Note! eliasbib.bib is always stored as latin1 not UTF8 

\usepackage{url}

\usepackage{floatrow}
\usepackage{subfig}

\pgfplotsset{compat=newest}
\pgfplotsset{plot coordinates/math parser=false}

\iftrue % set to iftrue if you want to regen PDF-files in gfx_tmp. 
  \usepackage{tikz}
  \usepackage{pgfplots}
  \pgfplotsset{
    tick label style={font=\scriptsize},
    label style={font=\scriptsize},
    legend style={font=\scriptsize}
  }
  \usetikzlibrary{decorations.pathreplacing}
  \iffalse % Set to iftrue if you want to regen fig PDF-files. 
     \usepgfplotslibrary{external}
  \else
     \renewcommand{\tikzsetnextfilename}[1]{}
  \fi
\else
  \usepackage{tikzexternal}
  \usepackage{pgfplots}
  \tikzsetexternalprefix{gfx_tmp/}
  
\fi
\newenvironment{thm}[1][\unskip]{\begin{theorem}[#1]}{\end{theorem}}
\newtheorem{remark}[theorem]{\textit{Remark}}

\newenvironment{lem}[1][\unskip]{\begin{lemma}[#1]}{\end{lemma}}

\title{The waveguide eigenvalue problem
and the tensor infinite Arnoldi method}
\author{Elias Jarlebring, Giampaolo Mele, Olof Runborg\thanks{Dept. Mathematics, KTH Royal Institute of Technology, SeRC swedish e-science research center, Lindstedtsv\"agen 25, Stockholm, Sweden, email: \{eliasj,gmele,olofr\}@kth.se}}
\date{\today}

\begin{document}
\maketitle

\begin{abstract}
%In this paper we construct a numerical method to 
%solve a nonlinear eigenvalue problem that arises in the
%study of waves propagation of waves in a periodic medium. 
We present a new computational approach for a class of large-scale 
nonlinear eigenvalue problems (NEPs) that are nonlinear
in the eigenvalue. The contribution of this paper is two-fold. 
We derive a new iterative algorithm for NEPs, 
the tensor infinite Arnoldi method (TIAR),
which is applicable to a general class of NEPs, and
we show how to specialize the algorithm to 
a specific NEP: the waveguide eigenvalue problem. 
The waveguide eigenvalue problem arises from a finite-element
discretization of a partial differential equation (PDE)
used in the study waves propagating in a periodic medium. 
%The size of the corresponding NEP depends on the domain discretization. 
The algorithm is successfully applied to accurately solve 
benchmark problems as well as complicated waveguides.
%The performance of the algorithm is illustrated with hund
We study the complexity of the specialized algorithm 
with respect to the number of iterations $m$ and the
size of the problem $n$, both from a theoretical perspective and
in practice. For the waveguide eigenvalue problem, we establish
that the computationally 
dominating part of the algorithm
has complexity $\mathcal{O}(nm^2+\sqrt{n}m^3)$.
Hence, the asymptotic complexity of TIAR applied to the
waveguide eigenvalue problem, for $n\rightarrow\infty$, is
the same as for Arnoldi's method for standard eigenvalue problems.

%here is the abstract
\end{abstract}
\section{Introduction}\label{sect:intro}
%In this paper we construct a numerical method to 
%solve a nonlinear eigenvalue problem that arises in the
%study of waves propagation of waves in a periodic medium. 
%
%\vskip 1 mm
%{\it + short sentence on the eigenvalue method}
%\vskip 1 mm
%
%We consider the propagation of time-harmonic electromagnetic waves 
Consider the propagation of 
waves  in a periodic medium,
which are governed by the Helmholtz equation 
%in periodic waveguides that are invariant in one dimension. The waves
%are then governed by the two-dimensional Helmholtz equation,
\begin{equation}\label{eq:helmholtz}
  \Delta v(x,z) + \omega^2 \eta(x,z)^2 v(x,z) = 0,\qquad (x,z)\in \RR^2,
\end{equation}
where $\eta \in L^\infty(\RR^2)$
is called the index of refraction and $\omega$ the temporal frequency.
When \eqref{eq:helmholtz} models an electromagnetic
wave, the solution $v$ typically represents the $y$-component of the electric or the magnetic field. % (depending on polarization). 
The (spatially dependent) 
wavenumber is  $\kappa(x,z):=\omega\eta(x,z)$
and we assume that the material is periodic in the $z$-direction and
without loss of generality the period is assumed to be 1,
i.e., 
$\eta(x,z+1)=\eta(x,z)$. 
The index of refraction is assumed to be 
constant for sufficiently large $|x|$, 
such that $\kappa(x,z)=\kappa_-$ when $x< x_-$, $\kappa(x,z)=\kappa_+$ when $x> x_+$.
%$\hat{n}(x,z)=\hat{n}_-$ when $x< x_-$ and 
%$\hat{n}(x,z)=\hat{n}_+$ when $x> x_+$. 
%, for some $\delta>0$.
In this paper we assume the wavenumber to be piecewise constant.
Figure~\ref{fig:wg_example} shows an example of the setup.

Bloch solutions to \eqref{eq:helmholtz} are those solutions 
that can be factorized as a product of a $z$-periodic function and
 $e^{\gamma z}$, i.e., 
\begin{equation}\label{eq:guidedform}
  v(x,z) = \hat{v}(x,z)e^{\gamma z},\qquad \hat{v}(x,z+1)=\hat{v}(x,z).
\end{equation}
The constant $\gamma\in\CC$ is called the Floquet multiplier and without loss of generality, 
it is assumed that $\Im\gamma\in(-2\pi,0]$.
%By the Floquet theorem, solutions to \eqref{eq:helmholtz} 
%can be factored into a $z$-periodic part and a phase factor,
We interpret \eqref{eq:helmholtz} in a weak sense.
We are only interested in Bloch solutions that
decay in magnitude as $|x|\rightarrow\infty$ and we require that 
$\hat{v}$, restricted to $S:=\RR\times (0,1)$, belongs to the Sobolev space $H^1(S)$. 
Moreover, we assume that any Bloch solution has a representative
in $C^1(S)$. These solutions are in general not in $C^2(S)$
since $\kappa$ is discontinuous.

%Note that if $\kappa$ is continuous almost everywhere,
\begin{figure}[ht]
  \begin{center}
    \scalebox{0.8}{\includegraphics{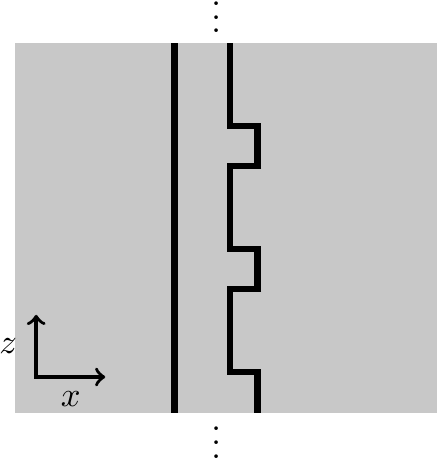}}
    \caption{Illustration of a waveguide defined on $\RR^2$. The wavenumber $\kappa(x,z)$ is constant in the three regions separated by the thick lines.}\label{fig:wg_example}
    \end{center}
\end{figure}
%the $H^1(S)$-solution to \eqref{eq:helmholtz}
%has a representive in $C^1(S)$, and in general not $C^2(S)$.
%In our setting $\hat{n}$ will have discontinuities and
% \eqref{eq:helmholtz} 
%therefore any  $H^1(S)$-solution to \eqref{eq:helmholtz}
%has a representive in $C_1(S)$ (in general not $C_2(S)$). 
In this contex, Bloch solutions are also called
%Solutions that satisfy \eqref{eq:guidedform} 
%are called 
guided modes of \eqref{eq:helmholtz}.
If $\gamma$ is purely imaginary, the mode is 
called {\em propagating}; if $|\Re\gamma|$ is small
it is called {\em leaky}.
Both mode types are of great interest in  various settings
\cite{Fliss:2013:DIRICHLET, Peng:1975:THEORY,Tausch:2013:SLABS,Stowell:2010:VARIATIONAL,Bao:1995:TIMEHARMONIC}.
We present a procedure to compute leaky modes, with $\Re \gamma <0$
and $\Im \gamma\in(-2\pi,0)$.
This specific setup has been studied, e.g., in
\cite{Tausch:2000:WAVEGUIDE}.

%In this paper we will focus on the leaky modes. 
%The problem problem has also been studied in \cite{Tausch:2000:WAVEGUIDE,Tausch:2010:OPEN}

%\vskip 1 mm
%{\it + rest of introduction from before}
%\vskip 1 mm
%
%
%Let consider the Helmholtz equation
%\begin{equation}\label{eq:helmholtz}
%   \Delta v(x,z)+\kappa(x,z)^2v(x,z)=0.
%\end{equation}
%We assume that the wavenumber satisfy 
%$\kappa(x,z)=\kappa_-$ when $x\le x_-$, 
%$\kappa(x,z)=\kappa_+$ when $x\ge x_+$ and is periodic in the direction $z$ with period, without loss of generality, 
%equal to $1$.
%This PDE arises in the study 
%of waves traveling in a periodic
%medium.  See  \cite{Tausch:2000:WAVEGUIDE} cite-more-waveguides.
%In order to study the properties of this problem \mytodo{Cite Tausch and the Floquet theorem}
%is needed to compute solutions of the form $v(x,z)=e^{\gamma z}\hat{v}(x,z)$, where 
%$\gamma \in \mathbb{C}$. The function $\hat{v}(x,z)$ is periodic in the direction $z$ with period $1$ and 
% decay to zero in the $x$ directions. Moreover we 
%are interested in computing such solutions 
%involving $\gamma$ which has negative real and imaginary part and 
%is close to the imaginary axis. 
%Then we reach the following infinite domain PDE--eigenvalue problem, 
%which we will call waveguide eigenvalue problem,
To compute the guided modes one can either fix $\omega$ and
find $\gamma$, or, conversely, fix $\gamma$ and find $\omega$.
Both formulations lead to a PDE eigenvalue problem set on 
the unbounded domain $S$. When $\gamma$ is held fix,  the eigenvalue
problem is linear and if $\omega$ is held fix, it is nonlinear
(quadratic) in $\gamma$. 
%
%However, in order to approximate
%the PDE numerically, the domain must be bounded. An absorbing
%boundary condition at some large enough value of $|x|$ is needed.
%In this paper we use a Dirichlet-to-Neuman map approach.
%This transforms both formulations into non-linear eigenvalue problems.
In this paper we fix $\omega$, 
and the substitution of \eqref{eq:guidedform}  into \eqref{eq:helmholtz}
leads to the following problem. %which we will solve numerically in this paper. 
Find $(\gamma,\hat{v})\in \CC\times H^1(S)$ such that
\begin{subequations}\label{eq:wg}
\begin{eqnarray}
 \;\;\;\;\;\;\;\; 
\Delta \hat{v}(x,z)&+&2\gamma \hat{v}_z(x,z)+
(\gamma^2+\kappa(x,z)^2)\hat{v}(x,z)=0,\;\;
(x,z)\in S,
\label{eq:wg0}	\\
\hat{v}(x,0)&=&\hat{v}(x,1),\;\;\textrm{}x\in\RR,\\
\hat{v}_z(x,0)&=&\hat{v}_z(x,1),\;\textrm{}x\in\RR.
%\hat{v}(x,\cdot)&\rightarrow& 0\textrm{ when }
%|x|\rightarrow\infty 
\end{eqnarray}
\end{subequations}
%%This PDE is obtained by considering
%%solutions to Helmholtz equation $\Delta v=\kappa(x,z)v$ with the
%%property $v(x,z)=e^{\gamma z}u(x,z)$, where $\kappa(x,z)$ is periodic in $z$ with period $1$. 
%
%
The problem \eqref{eq:wg}, which in this paper is referred to as the waveguide eigenvalue
problem,  is defined on an unbounded domain.
We use a well-known technique to reduce
the problem on a unbounded domain to a problem on a bounded domain. We impose 
artificial (absorbing) boundary conditions, in 
particular so-called Dirichlet-to-Neumann (DtN) maps.
See \cite{Hagstrom:2003:ABS,Berenger:1994:PML} for literature on artificial boundary conditions. 
%The finite-domain
%reformulation of \eqref{eq:wg}
%involves artificial boundary conditions expressed with
% so-called 
%

The DtN-reformulation and a finite-element discretization, with rectangular elements 
generated by a uniform grid with $n_x$ and $n_z$ grid points
in $x$ and $z$-direction correspondingly,
is presented in section~\ref{sect:discretization}.
A similar DtN-discretization has been applied to the waveguide 
eigenvalue problem in the literature \cite{Tausch:2000:WAVEGUIDE}.
In relation to \cite{Tausch:2000:WAVEGUIDE}, 
we need further equivalence results for the DtN-operator and 
use a different type of discretization, which allows easier
integration with our new iterative method.
Due to the fact that the DtN-maps depend on $\gamma$, 
the discretization leads to a nonlinear eigenvalue problem (NEP)
of the following type. Find $(\gamma,w)\in\CC\times\CC^n\backslash\{0\}$ such 
that
\begin{equation}
 M(\gamma)w=0,\label{eq:nep}
\end{equation}
where
\begin{equation} 
M(\gamma):=\begin{bmatrix}
  Q(\gamma) 		& 	C_1(\gamma)	\\
  C_2^T 		& 	P(\gamma)
\end{bmatrix}\in\CC^{n\times n},\label{eq:Mdef}
\end{equation}
and $n=n_xn_z+2n_z$.
The matrices 
$Q(\gamma)\in\CC^{n_xn_z\times n_xn_z}$ and 
$C_1(\gamma)\in\CC^{n_xn_z\times 2n_z}$
are a quadratic polynomials respect $\gamma$,
$Q(\gamma)	=	A_0	+	A_1\gamma	+ A_2\gamma^2$,
$C_1(\gamma)	=	C_{1,0} +	C_{1,1}\gamma	+ C_{1,2}\gamma^2$, 
%\begin{align}
% Q(\gamma)	&=	A_0	+	A_1\gamma	+ A_2\gamma^2,		\\
% C_1(\gamma)	&=	C_{1,0} +	C_{1,1}\gamma	+ C_{1,2}\gamma^2	
%\end{align}
where $A_i,C_{1,i}$ and $C_2^T$ are large and sparse.
%$\Lambda(\gamma)\in \CC^{2n_z\times 2n_z}$.
The matrix $P(\gamma)$ has the structure 
\begin{equation}\label{eq:Pmat}
  P(\gamma)=
\begin{bmatrix}
 R\Lambda_{-}(\gamma)R^{-1} & 0 \\
 0 &  R\Lambda_{+}(\gamma)R^{-1}
\end{bmatrix}\in\CC^{2n_z\times 2n_z}
\end{equation}
and $\Lambda_\pm(\gamma)\in\CC^{n_z\times n_z}$ are diagonal
matrices containing %square roots. 
nonlinear functions of $\gamma$ involving square roots of polynomials.
The  matrix-vector
product
corresponding to 
$R$ and $R^{-1}$ can be computed with 
the Fast Fourier Transform (FFT).

We have two main contributions in this paper:
\begin{itemize}
  \item a new algorithm, tensor infinite Arnoldi method (TIAR), 
for a general class of NEPs \eqref{eq:nep}, which is based on a tensor 
representation of the basis of the infinite Arnoldi method (IAR) \cite{Jarlebring:2012:INFARNOLDI};
  \item an adaption of TIAR to the structure  arising from a particular type of discretization of the waveguide eigenvalue problem.
\end{itemize}
%a new algorithm 
%for NEPs, which we call Tensor Infinite Arnoldi (TIAR) and 
%an efficient adaption of this to the problem 
%\eqref{eq:Mdef}.
%

%* literature overview below needs to be rewritten * 
The general NEP \eqref{eq:nep} 
has received considerable attention in the literature
in various generality settings. 
We list those algorithm that are related
to our method. 
See the review papers
\cite{Mehrmann:2004:NLEVP,Voss:2013:NEPCHAPTER}
and the problem collection \cite{Betcke:2013:NLEVPCOLL}, for
further literature.

Our algorithm TIAR is based on \cite{Jarlebring:2012:INFARNOLDI},
but uses a compact representation of 
Krylov subspace generated by a particular structure 
in the basis matrix.
Different compact representations for iterative
methods for polynomial eigenvalue problems and other NEPs
have been developed in other works. In particular,
the basis matrices stemming from Arnoldi's method
applied to a companion linearizations of polynomial  
eigenvalue problems can be exploited
by using reasoning with the Arnoldi factorization. 
This has been done for Arnoldi methods in \cite{Kressner:2014:TOAR,Zhang:2013:MOR,Bai:2005:SOAR} and rational Krylov methods \cite{VanBeeumen:2014:CORK}. The approach \cite{Kressner:2014:TOAR} is 
designed for polynomial eigenvalue problems expressed in a Chebyshev basis. 
 makes it 
particularly suitable to use in a two stage-approach,
which is done in \cite{Effenberger:2012:CHEBYSHEV}, where
the eigenvalues of interest lie in a predefined interval 
and (a non-polynomial) $M$ can first be approximated with interpolation 
on a Chebyshev grid and subsequently the 
polynomial eigenvalue problem can be solved with \cite{Kressner:2014:TOAR}. 
The algorithm in \cite{Zhang:2013:MOR} is mainly developed for moment-matching in model  reduction of time-delay systems, where the main goal is to compute a subspace (of $\CC^n$) with appropriate approximation properties. 
We stress that the algorithm in the preprint \cite{VanBeeumen:2014:CORK}, 
which has been developed in parallell independent of our work, 
is similar to TIAR in the sense that it 
can be interpreted as a rational Krylov method for
general NEPs involving a compact representation. 
The derivation is very different 
(involving reasoning with a compact Krylov factorization as in \cite{Zhang:2013:MOR,Kressner:2014:TOAR}), 
and leads to a general class of methods with different algorithmic properties.

Some recent approaches for \eqref{eq:nep} exploit
low-rank properties, e.g., 
$M^{(i)}(0)=V_iQ^T$, where $V_i,Q\in\CC^{n\times r}$
for sufficiently large $i$, 
and $r$ is small relative to $n$.
See, e.g., \cite{Su:2011:REP,Voss:2011:LOWRANK,VanBeeumen:2014:LOWRANK}.
This property is present here if we select
$r=n_z=\mathcal{O}(\sqrt{n})$, which is not very small 
with respect to the size of the problem, making
the low-rank methods to not appear favorable for this NEP.
%Unlike these, our approach involves a representation 
%with tensors and the approximate eigenvalues 
%are computed from the eigenvalues of a Hessenberg matrix.
%We stress that  the  \cite{VanBeeumen:2014:CORK} was 

The (non-polynomial) nonlinearities
in our approach stem from absorbing boundary
conditions. Other absorbing boundary conditions 
also lead to NEPs. This has been illustrated
in specific applications, e.g., in 
the simulation of optical fibers \cite{Kaufman:2006:FIBER},
cavity in accelerator design \cite{Liao:2010:NLRR},
double-periodic photonic crystals
\cite{Effenberger:2012:LINEARIZATION,Engstrom:2014:SPECTRAL}
and microelectromechanical systems  \cite{Bindel:2005:ELASTIC}.
There is to our knowledge no approach that integrates
the structure of the discretization
of the PDE and the $\gamma$-dependent 
boundary conditions with an Arnoldi method.
The adaption of the algorithm to our specific
PDE is presented in section~\ref{sect:adaption}.

The notation is mostly standard. A matrix consisting of elements 
$a_{i,j}$ is denoted
\[
[a_{i,j}]_{i,j=1}^m=\begin{bmatrix}
a_{1,1} & \cdots & a_{1,m}\\
\vdots &  & \vdots\\
a_{m,1} & \cdots & a_{m,m}
\end{bmatrix}.
\]
The notation is analogous for vectors and tensors.
We use $\underline{Q}$ to denote 
an extension of $Q$ with one block row of zeros. 
 The size of the block will be clear by the context. 

%We use the notation $[a_i]_{i=1}^m\in\CC^m$ for a vector defined analogously.

%* We will denote the set of 
%square integrable functions $f:\Omega\rightarrow\CC$
%by  $L^2(\Omega)$. 

%* Notation: $\diag(a_1,\ldots a_m)$ will corresond
%to a diagonal matrix with diagonal elements $a_1,\ldots,a_m$.

\section{Derivation of the NEP}\label{sect:discretization}

%\subsection{Regularity of the solutions and structure of the spectrum}

%In the equation \eqref{eq:helmholtz} we are interested 
%in the solutions which lie in $H^1(\Omega)$ for 
%any $\Omega \subset \RR^2$. Thought classical arguments of 
%functional analysis and the Sobolev embedding theorem, 
%it is straightforward to show that such solutions have 
%a representative in $C^1(\Omega)$. This property 
%is naturally inherit by the solutions of the 
%eigenvalue problem 
%\eqref{eq:wg}. In this paper we 
%will compute such solutions. Moreover if
%$v(x,z)=e^{\gamma z}\hat{v}(x,z)$ satisfy \eqref{eq:helmholtz}, also 
%$v(x,z)=e^{\gamma+2 \pi z}\hat{v}(x,z)$ satisfy \eqref{eq:helmholtz}. 
%Therefore we can restrict our attention to the $\gamma$ with 
%imaginary part between $0$ and $-2 \pi$.
%
%
%

\subsection{DtN reformulation}

As a first step in deriving a computational
approach to \eqref{eq:wg}, we
 rephrase the problem on a bounded domain
$S_0:=[x_-,x_+]\times [0,1]\subset S$
by introducing artificial boundary conditions at $x=x_\pm$.
We use a construction with so-called
Dirichlet-to-Neumann (DtN) maps
which  relate the (normal) derivative 
of the solution at the boundary with the function value at the boundary.
The main concepts of DtN maps
are presented in  various generality settings in, e.g., 
\cite{Tausch:2000:WAVEGUIDE,Keller:1989:EXACT,Harari:1998:DIRICHLET,Fliss:2013:DIRICHLET,Givoli:2004:DISPERSIVE}.
We use the same DtN maps as in \cite{Tausch:2000:WAVEGUIDE}, but we 
use a different discretization and we need
to derive
some results necessary for our setting.

The DtN formulation of the eigenvalue problem \eqref{eq:wg} 
is given as follows. Find $\gamma$ and $u\in H^1(S_0)$
such that
\begin{subequations}\label{eq:wg_finite}
\begin{align}
     \Delta u(x,z)+2\gamma u_z(x,z)+(\gamma^2+\kappa(x,z)^2)u(x,z)
     &=0,
&(x,z)&\in S_0,\\
u(x,0)&=u(x,1), &x&\in(x_-,x_+),\\
u_z(x,0)&=u_z(x,1),  & x&\in(x_-,x_+),\\
\mathcal{T}_{-,\gamma}[u(x_-,\cdot)]&=-u_x(x_-,\cdot),{\hskip -10 mm}&&\label{eq:wg_finite_T1}\\
\mathcal{T}_{+,\gamma}[u(x_+,\cdot)]&=u_x(x_+,\cdot),{\hskip -10 mm}&&\label{eq:wg_finite_T2}
\end{align}
\end{subequations}
where $\mathcal{T}_{\pm,\gamma}:H^{1}([0,1])\mapsto L^2([0,1])$
are the DtN maps, defined by 
%
% $g\in H^1({\mathbb T})$ they are defined by
\begin{equation}\label{eq:DtNdef}
\mathcal{T}_{\pm,\gamma}[g](z):=\sum_{k\in{\mathbb Z}} 
%i\sqrt{\beta_{\pm,k}}
s_{\pm,k}(\gamma)
g_k e^{2\pi i k z},
\end{equation}
where $[g_k]_{k\in \ZZ}$ is the Fourier expansion of $g$, i.e., 
$g(z):=\sum_{k\in{\mathbb Z}} g_k e^{2\pi i k z}$ and
%\quad\Rightarrow\quad
\begin{eqnarray}
  \beta_{\pm,k}(\gamma) &:=& 
(\gamma+2i\pi k)^2+\kappa^2_\pm=((\gamma+2i\pi k)+i\kappa_\pm)((\gamma+2i\pi k)-i\kappa_\pm)
,\label{eq:betadef}\\
  s_{\pm,k}(\gamma)&:=&
\sign(\im\left(\beta_{\pm,k}(\gamma)\right))i\sqrt{\beta_{\pm,k}(\gamma)}.\label{eq:sk}
\end{eqnarray}
%\begin{equation}\label{eq:betadef}
%\end{equation}
%and the sign of the square root is chosen such that $\Im\sqrt{\beta_k}>0$.
In this section we show that,
under the assumption that neither the real nor the imaginary part of $\gamma$
vanish,
the DtN maps are well-defined and
the problems \eqref{eq:wg} 
%(with $\omega n(x,z)=\kappa(x,z)$)
 and \eqref{eq:wg_finite}
are equivalent. 
In order to characterize the
DtN maps, we consider the exterior problems,
i.e., the 
problems corresponding to the 
domains
$S_+ =(x_+,\infty)\times (0,1)$
and $S_- =(-\infty,x_-)\times (0,1)$.
The exterior problems are defined
as the two problems corresponding
to finding $w\in H^1(S_\pm)$ 
such that, for a given $g$,
%The $S_+$ problem is
% find $u\in H^1(S_+)$ such that
\begin{subequations}\label{eq:wg_exterior}
\begin{align}
     \Delta w+2\gamma w_z+(\gamma^2+\kappa_\pm^2)w&=0, 	&	(x,z)\in S_\pm, \label{eq:wg_exterior1}\\
        w(x,0)&=w(x,1),     \\%\qquad     & x&>x_+,\\
        w_z(x,0)&=w_z(x,1), \\%\qquad     & x&>x_+,\\
        w(x_\pm,z)&=g(z). %\qquad     & z&\in(0,1).
\end{align}
\end{subequations}
%The problem corresponding
%to $S_- =(-\infty,x_-)\times (0,1)$
% is defined analogously.
\begin{remark}[Regularity]\label{remark:smooth}
%By elliptic regularity \cite[Section 6.3.1, Theorem 1]{Evans:2010:PDEs},
%weak $H^1$ solutions of 
Note that if we multiply a solution 
to \eqref{eq:wg},
\eqref{eq:wg_finite} or \eqref{eq:wg_exterior}
%are in $H^2_{\rm loc}$ and, if they are 
with $e^{\gamma z}$, we have a solution to the Helmholtz
equation, i.e., it satisfies \eqref{eq:helmholtz} in their respective 
domains, i.e. $S$, $S_0$ and $S_\pm$. 
%We assumed that the solutions of \eqref{eq:helmholtz} are in
%$C^1$. This means that a solution 
%to \eqref{eq:wg},
%\eqref{eq:wg_finite} or \eqref{eq:wg_exterior} is $C^1$ and 
By assumption, solutions 
to \eqref{eq:wg},
\eqref{eq:wg_finite} and \eqref{eq:wg_exterior}  are $C^1$ and 
the traces taken on $x=x_\pm$
and its first derivatives
are always well-defined and continuous. %smooth. 
Moreover, for $x<x_-$
%--\delta$ 
and 
$x>x_+$,
%+\delta$,
since $\kappa(x,z)$ is constant, the problem
can be interpreted in a strong sense and the solutions
are in $C^\infty$. 

Our assumption that the solution has regularity $C^1$ can 
be relaxed as follows. If we select $x_-$ and $x_+$ such that
 $\kappa$ is constant over $x_-$ and $x_+$, we have 
by elliptic regularity \cite[Section 6.3.1, Theorem 1]{Evans:2010:PDEs}, that
weak $H^1$ solutions of \eqref{eq:wg},
\eqref{eq:wg_finite} and \eqref{eq:wg_exterior}
are in $H^2_{\rm loc}$. This means that traces taken on $x=x_\pm$
of a solution and its derivatives
are always well-defined and smooth, without explicitly assume 
that the solution is in $C^1$.

%+\delta$,
%This means that traces taken on $x=x_\pm$
%of a solution and its derivatives
%are always well-defined and smooth. 
%* $C_1$ as stated
%in the introduction * 

%* must be integrated *
%In the equation \eqref{eq:helmholtz} we are interested 
%in the solutions which lie in $H^1(\Omega)$ for 
%any $\Omega \subset \RR^2$. Thought classical arguments of 
%functional analysis and the Sobolev embedding theorem, 
%it is straightforward to show that such solutions have 
%a representative in $C^1(\Omega)$. This property 
%is naturally inherit by the solutions of the 
%eigenvalue problem 
%\eqref{eq:wg}. In this paper we 
%will compute such solutions. Moreover if
%$v(x,z)=e^{\gamma z}\hat{v}(x,z)$ satisfy \eqref{eq:helmholtz}, also 
%$v(x,z)=e^{\gamma+2 \pi z}\hat{v}(x,z)$ satisfy \eqref{eq:helmholtz}. 
%Therefore we can restrict our attention to the $\gamma$ with 
%imaginary part between $0$ and $-2 \pi$.
%

\end{remark}
The following result illustrates that the application of the DtN maps 
in \eqref{eq:DtNdef} is in a sense 
equivalent to solving the exterior problems and evaluating
the solutions in the normal direction at
the boundary $x=x_\pm$.
%The key property of the DtN maps is that they map
%the Dirichlet condition $g$ on the boundary $x=x_\pm$ to the normal
%derivative of the solution $\pm u_x(x_\pm,z)$ on the same boundary, 
%as long as the exterior problems are well-posed.
More precisely, the following lemma shows that
if $\Re \gamma\neq 0$ and $\Im \gamma\in(-2\pi,0)$ the problems are well-posed in $H^1(S_\pm)$ and 
the boundary relations \eqref{eq:wg_finite_T1} and
\eqref{eq:wg_finite_T2} are satisfied. The proof
is available in Appendix~\ref{sect:DtNproof}.
\begin{lemma}[Characterization of DtN maps] \label{thm:DtN}
Suppose $\Re\gamma\neq 0$,
%,  $\Im\gamma\neq 0$ 
and $\Im\gamma\not\in 2\pi\ZZ$ 
%$\Im \gamma\in(0,2\pi)$
 and $g\in H^{1/2}([0,1])$. 
Then, each of the  exterior problems \eqref{eq:wg_exterior}
have a unique solution $w\in H^1(S_\pm)$.
 Moreover,
there is a constant $C$ independent of $g$ such that
\begin{equation}\label{eq:uest}
  ||w||_{H^1(S_\pm)}\leq C ||g||_{H^{1/2}([0,1])}.
\end{equation}
%If the DtN-maps are well-defined by \eqref{eq:DtNdef} they satisfy
If it is further assumed that $g\in H^{1}([0,1])$, then 
the DtN maps in \eqref{eq:DtNdef} are well-defined and satisfy
\begin{equation}\label{eq:DtNrel}
{\mathcal T}_{+,\gamma} [w(x_+,\cdot)](z)=
w_x(x_+,z),\qquad
{\mathcal T}_{-,\gamma} [w(x_-,\cdot)](z)=
-w_x(x_-,z).
\end{equation}
%and they are well-defined if $g\in H^{1}([0,1])$.
\end{lemma}

This lemma immediately implies the equivalence between
\eqref{eq:wg} and \eqref{eq:wg_finite} under 
the same conditions on $\gamma$. %We get

\begin{thm}[Equivalence of \eqref{eq:wg} and \eqref{eq:wg_finite}]
Suppose $\Re\gamma\neq 0$ and
%,  $\Im\gamma\neq 0$ 
% $\Im \gamma\in(0,2\pi)$.
$\Im \gamma \not\in 2\pi\ZZ$.
Then $u\in H^1(S_0)$ is a solution to \eqref{eq:wg_finite}
if and only if there exists a solution $\hat{v}\in H^1(S)$
to \eqref{eq:wg} such that $u$ is the restriction of $\hat{v}$ to $S_0$.
\end{thm}
\begin{proof}
Suppose $\hat{v}$ is a solution of \eqref{eq:wg}
and $u$ is its restriction to $S_0$.
Then $u$ clearly satisfies (\ref{eq:wg_finite}a-c).
By Remark~\ref{remark:smooth} the functions $\hat{v}(x_\pm,z)=u(x_\pm,z)$
are in $C^1([0,1])\subset H^1([0,1])$. 
%and $\hat{v}_x(x_\pm,z)=u_x(x_\pm,z)$ are
%in $C_0([0,1])\subset H_1([0,1])$.
Lemma~\ref{thm:DtN} shows that 
$\hat{v}$, restricted to $S_{\pm}$, are the unique
solutions to the exterior problems \eqref{eq:wg_exterior}.
Hence, $\hat{v}$ is identical to the union of $u$ and 
 the solutions to the exterior problems \eqref{eq:wg_exterior}. %$w \in H^1(S_\pm)$. 
 Since $\hat{v}\in C^1(S)$, we have
that $\hat{v}_x(x_\pm,z)$ is continuous and $w_x(x_\pm,z)=u_x(x_\pm,z)=\hat{v}_x(x_\pm,z)$. 
Moreover, due to \eqref{eq:DtNrel}, 
the boundary conditions (\ref{eq:wg_finite}d-e) 
are satisfied.
 
%Since, \eqref{eq:DtNrel} relates the one-sided derivatives
%of $w$ 
%, and
%${\mathcal T}_{\pm,\gamma}u(x_\pm,z)={\mathcal T}_{\pm,\gamma}\hat{v}(x_\pm,z)
%=\pm \hat{v}_x(x_\pm,z)=\pm u_x(x_\pm,z)$.
%Hence, $u$ also satisfies (\ref{eq:wg_finite}d-e)
%and it thus solves \eqref{eq:wg_finite}.

On the other hand, suppose $u\in H^1(S_0)$ is 
a weak solution to \eqref{eq:wg_finite}.
Remark~\ref{remark:smooth} again implies that 
$u\in C^1(S_0)$ and in particular 
$u(x_\pm, z)\in C^1([0,1])\subset H^{1}([0,1])$.
%The Sobolev trace theorem implies that $u(x_\pm,\cdot)\in H^{1/2}([0,1])$,
We have from Lemma~\ref{thm:DtN} 
that the exterior problems \eqref{eq:wg_exterior} have unique solutions $w$ that satisfy \eqref{eq:DtNrel}.
%we let $w_{\pm}$ 
%be the solutions of the exterior problems
%(which are unique since \eqref{eq:wg_finite_T1}
%and \eqref{eq:wg_finite_T2} are satisfied)
%with
%boundary conditions $g(z)=u(x_{\pm},z)$ 
%(is again smooth by Remark~\ref{remark:smooth})
%and construct $\hat{v}$ 
%as the union of $u$ and $w_{\pm}$. 
Let $\hat{v}$ be defined as the union of the $u$ and $w$. 
The union $\hat{v}$ has a continuous derivative
on the boundary $x=x_\pm$ due to \eqref{eq:wg_finite}d-e and \eqref{eq:DtNrel} and 
since $u\in H^1(S_0)$ and $w\in H^1(S_\pm)$, then 
$\hat{v}\in H^1(S)$
and satisfies \eqref{eq:wg} by construction.
%Then by Theorem~\ref{thm:DtN} we have
%$\pm u_x(x_\pm,z) = 
%{\mathcal T}_{\pm,\gamma}u(x_\pm,z)={\mathcal T}_{\pm,\gamma}\hat{v}_\pm(x_\pm,z)
%=\pm \partial_x \hat{v}_\pm(x_\pm,z)$ 
%showing that $\hat{v}_x$ is continuous across $x=x_\pm$
%and $\hat{v}$ is therefore a $H^1$ solution in all of $S$.
\end{proof}

\begin{remark}[Conditions on $\gamma$]
Modes with $\Re\gamma=0$ are propagating. 
For those modes, the well-posedness
of the DtN-maps depends on the wave number.  
%the DtN operators are not always well-defined. 
%may or may not 
%be well-defined, and
%\eqref{eq:wg} may or may not be equivalent to \eqref{eq:wg_finite}.
%This depends on $\omega$ and the band structure of the problem's
%spectrum. 
See \cite{Fliss:2013:DIRICHLET} for precise results about
well-posedness in the situation $\Re\gamma=0$.
In our setting we only consider leaky modes and $\re\gamma<0$.
The situation $\re\gamma>0$ can be treated analogously.
%Modes with $\Im \gamma=0$ correspond to waves propagating perpendicular to the waveguide. {\bf (?) Maybe there is a better characterization.} "Non-oscillatory
%in the direction of the waveguide." (?)
\end{remark}

\subsection{Discretization}\label{sect:intdisc}
We discretize the finite-domain PDE \eqref{eq:wg_finite} 
with a finite-element approach. 
The domain $[x_-,x_+]\times[0,1]$ is partitioned using rectangular elements 
obtained with 
a uniform distribution of nodes in the $x$ and $z$ directions.
We use $n_x$ grid points in the $x$-direction
and $n_z$ grid points in the $z$-direction and define 
$x_i=x_-+i h_x$ and
$z_j=jh_z$ 
where 
$i=1,\ldots,n_x$,
$j=1,\ldots,n_z$,
$h_x=\frac{x_+-x_-}{n_x+1}$ and
$h_z=\frac{1}{n_z}$. 
%The generic element is $[x_i, x_{i+1}] \times [z_j, z_{j+1}]$. 
The basis functions are chosen as piecewise bilinear functions that are periodic in the 
$z$ direction with period $1$. In particular, the basis functions that we consider are
 periodic modification of the standard basis functions. 
We denote them as
% recall that we are sorting not in the usual way the basis functions
$\left \{ \phi_{i,j}(x,z) \right \}_{i=1,j=1}^{n_x,n_z},\left \{ \phi_{\pm,j}(x,z) \right \}_{j=1}^{n_z}$ 
where 
$\phi_{i,j}(x_t,z_s) = \delta_{i,s} \delta_{j,t}$, 
$\phi_{-,j}(x_-,z_s) = \delta_{j,s}$ and 
$\phi_{+,j}(x_+,z_s) = \delta_{j,s}$.
The finite-domain PDE \eqref{eq:wg_finite} can be rewritten in weak form
\begin{equation}	\label{eq:weak_PDE}
 a(u, \phi) + \gamma b(u,\phi) + 
\gamma^2 c(u,\phi)  = 0
\hspace{0.5cm}	\forall \phi \in H^1(S_0)
\end{equation}
where $a,b$ and $c$ are bilinear operators. 
The approximation of a solution of \eqref{eq:weak_PDE} 
can be expressed as 
\begin{equation} \label{eq:solution_FEM}
 \tilde{u}(x,z) = 
 \sum_{i=1}^{n_x}  \sum_{j=1}^{n_z} u_{i,j} \phi_{i,j} (x,z) 		+
 \sum_{j=1}^{n_z}  u_{-,j} \phi_{-,j} (x,z)					+
 \sum_{j=1}^{n_z}  u_{+,j} \phi_{+,j} (x,z).			
\end{equation}
We represent the coefficients that defines $\tilde{u}(x,z)$ in a compact way
\[
\hat{u} = \operatorname{vec}
\begin{bmatrix}
u_{1,1} 		& 	\cdots 	& 	u_{n_x,1}  	\\
  \vdots 		&  		&	\vdots  	\\
u_{1,{n_z}}		& 	\cdots  & 	u_{n_x,n_z} 	\\
\end{bmatrix},\;
\hat{u}_-=
\begin{bmatrix}
u_{-,1}	\\
\vdots 		\\
u_{-,n_z}	\\
\end{bmatrix},\;
\hat{u}_+=
\begin{bmatrix}
u_{+,1}	\\
\vdots 		\\
u_{+,n_z}	\\
\end{bmatrix}.
\]
such that the Ritz--Galerkin discretization of \eqref{eq:weak_PDE} leads to the following relation 
\begin{equation}
Q(\gamma) \hat{u} +
C_1(\gamma)	\begin{bmatrix}
		\hat{u}_-\\
		\hat{u}_+
		\end{bmatrix}=0,
\label{eq:intdisc}
\end{equation}
where 
\begin{align*}
Q(\gamma)	& :=	A_0+\gamma A_1+\gamma^2A_2,			\\
C_1(\gamma)	& :=	C_{1,0}+\gamma C_{1,1}+\gamma^2 C_{1,2}.	
\end{align*}
The matrices $(A_i)_{i=0}^2$ and $(C_{1,i})_{i=0}^2$ can be computed in an efficient and explicit way\footnote{The 
matrices are available online in order to make the
results easily reproducible:
\url{http://people.kth.se/~gmele/waveguide/}} with the 
procedure outlined in Appendix~\ref{sect:fem}.

Two approximations must be done in order to 
incorporate the boundary conditions.
%of \eqref{eq:intdisc}.
We construct approximations of
the right-hand side of \eqref{eq:wg_finite}d-e 
using the one-sided second-order finite-difference approximation,
\begin{equation} \label{eq:oneside_approx}
\begin{bmatrix} 
  - u_x(x_-,z_1)\\
  \vdots\\
  - u_x(x_-,z_{n_z})
\end{bmatrix}
\approx  -C_{2,-}^T\hat{u}-d_0\hat{u}_-  
\textrm{ and  }
\begin{bmatrix}
  u_x(x_+,z_1)\\
  \vdots\\
  u_x(x_+,z_{n_z})
\end{bmatrix}
\approx  -C_{2,+}^T\hat{u}-d_0\hat{u}_+  
\end{equation}
where $C_{2,-}^T=
(d_1,d_2,0,\ldots,0)\otimes I_{n_z}\in\CC^{n_z\times n_zn_x}$
and $C_{2,+}^T= 
((0,\ldots,0,d_2,d_1)\otimes I_{n_z})\in\CC^{n_z\times n_zn_x}
$ with  $d_0=-\frac{3}{2h_x}$, $d_1=\frac{2}{h_x}$ and  $d_2=-\frac{1}{2h_x}$.
The DtN maps in the left-hand side of (\ref{eq:wg_finite}d-e)
act on the function values on the
 boundary only, i.e., the function approximated by $\hat{u}_\pm$. 
We compute the first $p$ Fourier coefficients
of the approximated function, apply the definition of 
$\mathcal{T}_{\pm,\gamma}$
on the Fourier coefficients,
and convert the Fourier expansion back to 
the uniform grid. 
More precisely,
the approximation of
the left-hand side of \eqref{eq:wg_finite}d-e
is given by
\begin{equation}\label{eq:Tapprox}
\left[\left.\mathcal{T}_{\pm,\gamma}(u(x_\pm,z))\right|_{z=z_i}\right]_{i=1}^{n_z}
\approx RL_{\pm}(\gamma)R^{-1}\hat{u}_\pm\in\CC^{n_z}
\end{equation}
where 
 $L_\pm(\gamma)=\diag([s_{j,\pm}(\gamma)]_{j=-p}^p)$, and 
$R=[\exp(2i\pi j z_k)]_{k=1,j=-p}^{n_z,p}$ 
with $n_z=2p+1$. In the algorithm we exploit
that the action of $R$ and $R^{-1}$ 
can be computed with FFT. 
We match \eqref{eq:Tapprox} and \eqref{eq:oneside_approx}
and get a discretization of the boundary condition
\eqref{eq:wg_finite}d-e. 
That is, we reach
the NEP \eqref{eq:nep},
with $M$ given by \eqref{eq:Mdef}
if
we define
$C_2:=\begin{bmatrix}C_{2,-} & C_{2,+}\end{bmatrix}$,
$\Lambda_{\pm}(\gamma):=L_{\pm}+d_0 I$
and $w^T:=\begin{bmatrix}\hat{u}^T & \hat{u}_-^T & \hat{u}_+^T\end{bmatrix}$
and combine \eqref{eq:intdisc}
with \eqref{eq:Tapprox} and \eqref{eq:oneside_approx}.

\section{Derivation and adaption of TIAR}\label{sect:adaption}
%In this section we consider a generic nonlinear 
%eigenvalue problem $T(\lambda)w=0$ that do not necessary 
%correspond to the problem~\eqref{eq:Mdef}. 
%Therefore, 
%the algorithm that will derive can be used 
%also for other problems.
\subsection{Basis matrix structure of the infinite Arnoldi method (IAR)}
There exists several variations of IAR, \cite{Jarlebring:2010:DELAYARNOLDI,Jarlebring:2012:INFARNOLDI}.
%The most general form of IAR was presented in  \cite{Jarlebring:2012:INFARNOLDI}. See also the earlier version for the delay 
%eigenvalue problem in \cite{Jarlebring:2010:DELAYARNOLDI}.
% PUT IN INTRO
% It is an algorithm that can be interpreted as a 
% standard (linear) eigenvalue algorithm (the Arnoldi method)
% applied
% to an infinite-dimensional operator and 
% it can be carried out with operations
% on vectors and matrices of size $n$. 
% From the fact that the algorithm is equivalent
% to the Arnoldi method, it is expected to 
% have the same attractive
% convergence properties, e.g., high reliability.
 We use the
variant of IAR
in \cite{Jarlebring:2012:INFARNOLDI}
called the Taylor variant, as it is based on the
Taylor coefficients (derivatives) of $M$.
We briefly summarize the algorithm
and
characterize a structure in the basis matrix.
Similar to the standard Arnoldi method, IAR is
an algorithm with an algorithmic state
consisting of a basis matrix $Q_k$ and a Hessenberg
matrix $H_k$. The
basis matrix and the Hessenberg matrix are expanded in every loop.
Unlike the standard Arnoldi method, in IAR, 
the basis matrix is 
expanded by a block row as well as a 
column,
leading to a basis matrix with block triangular structure, 
where the leading (top left) submatrix of the basis 
matrix is the basis matrix of the previous loop. 
More precisely, there 
exist vectors  $q_{i,j}\in\CC^n$, 
$i,j=1,\ldots,k$ such that 
\begin{equation}  \label{eq:Vkblocks}
 Q_k=
 \begin{bmatrix}
   q_{1,1} & q_{1,2} & \cdots   & q_{1,k}  \\
     0    &  q_{2,2}&      &  \\
     \vdots     & \ddots &   \ddots   &\vdots  \\
     0     &  \cdots   & 0 & q_{k,k} 
 \end{bmatrix}.
\end{equation}

In every loop in IAR
we must compute a new vector to be used in the expansion of $Q_k$ and $H_k$.
In practice, in iteration $k$, this reduces
to computing $y_1\in\CC^n$ given $y_2,\ldots,y_{k+1}$ 
such that 
\begin{equation} \label{eq:y0}
  y_1=-M(0)^{-1}\left(
\sum_{i=1}^{k}
M^{(i)}(0)
y_{i+1}
\right).
\end{equation}
Clearly, since $M(0)$ does
not change throughout the iteration, and
we can compute an LU-factorization before
starting the algorithm, such that the linear system can be
solved efficiently in every iteration. IAR (Taylor version) 
is for completeness given by algorithm~\ref{alg:infarnoldi}.

\begin{algorithm} %\SetLine %% new algorithm2e: \SetAlgoLined
\caption{Infinite Arnoldi method - IAR (Taylor version) \cite{Jarlebring:2012:INFARNOLDI}\label{alg:infarnoldi}}
\SetKwInOut{Input}{input}\SetKwInOut{Output}{output}
\Input{$x_1\in\CC^n$}
\BlankLine
\nl Let $Q_1=x_1/\|x_1\|_2$,  $H_0=$empty matrix\\
\For{$k=1,2,\ldots,m$ }{
\nl Compute $y_2,\ldots,y_{k+1}$  from the last column of $Q_k$ 
by setting 
\[
  y_j=\frac{1}{j-1}q_{j-1,k}
\]
for $j=2,\ldots,k+1$.\\
\nl Compute $y_1$ from $y_2,\ldots,y_{k+1}$
using \eqref{eq:y0}\\
\nl Let $y:=\vect(y_1,\ldots,y_{k+1})$ and  $\underline{Q}_k:=\begin{bmatrix}
Q_k\\0\end{bmatrix}\in\CC^{(n(k+1))\times k}$.\label{step:infarn:y}\\
\nl Compute $h=\underline{Q}_k^H y$\label{step:infarn:h}\\
\nl Compute $y_{\perp} = y - \underline{Q}_k h$\label{step:infarn:yperp}\\
\nl Possibly repeat Steps \ref{step:infarn:h}-\ref{step:infarn:yperp} and get new $h$ and $y_\perp$\\
\nl Compute $\beta = \|y_\perp\|_2$ \\
\nl Compute $q_{k+1} = y_\perp / \beta$\\
\nl Let $H_k = \left[\begin{array}{cc}H_{k-1} & h \\ 0 & \beta \end{array}\right] \in\CC^{(k+1)\times k}$\\
\nl Expand $Q_k$ into $Q_{k+1} = [\underline{Q}_k,q_{k+1}]$\\
}
\nl Compute the eigenvalues $\{\mu_i\}_{i=1}^m$ of the leading $m\times m$ submatrix
of the Hessenberg matrix $H_k$	\\
\nl return approximations $\{1/\mu_i\}_{i=1}^k$\\
\end{algorithm}

\begin{figure}[ht]
  \begin{center}
    %\subfigure[mytext]{\scalebox{0.65}{\includegraphics{file}}}
    \scalebox{0.8}{\includegraphics{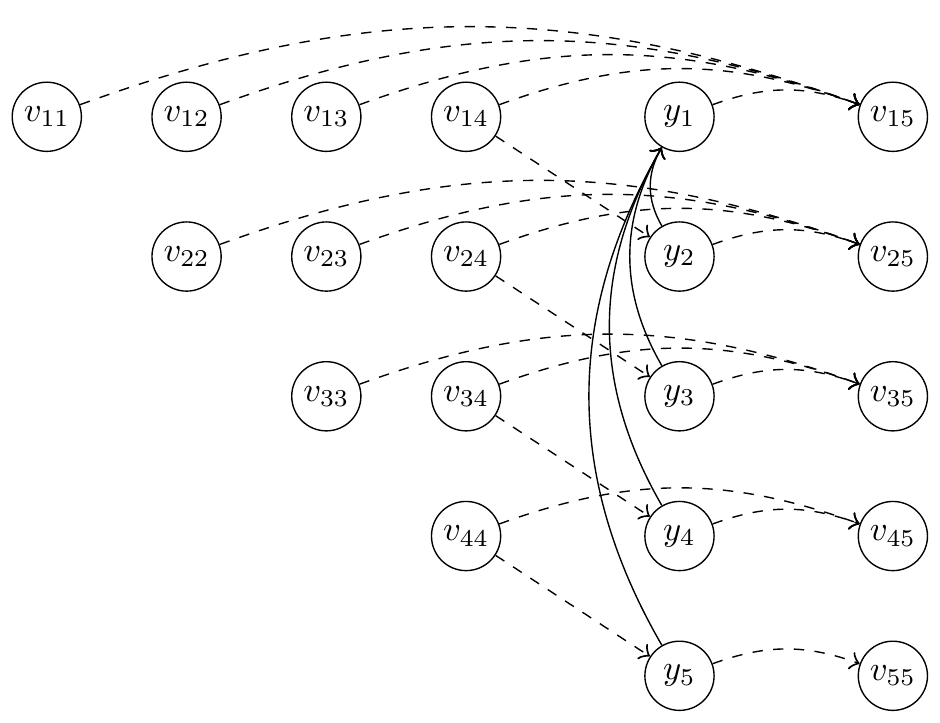}}
    \caption{ 
        The computation tree 
representing Steps 3-9 for Algorithm~\ref{alg:infarnoldi}
when $k=4$, i.e., after three iterations. 
Every node is vector of size $n$. 
The dashed lines correspond to computing linear
combinations. Clearly,
the only (potentially) new direction 
of the span of all vectors can be represented by $y_1$. 
\label{fig:comp_tree}
    }
  \end{center}
\end{figure}

Steps 3-9 of Algorithm~\ref{alg:infarnoldi}
are visualized in Figure~\ref{fig:comp_tree} when $k=4$, i.e., after three iterations. We have 
marked those operations that are linear combinations as
dashed lines. The fact that
the many operations are linear combinations leads to a structure in $Q_k$ which can be exploited 
such that we can reduce the 
usage of computer resources (memory and computation time) 
and maintain an equivalence with Algorithm~\ref{alg:infarnoldi}. 

More precisely, the block elements of the basis matrix $Q_k$ have
 the following structure.
\begin{lem}[Structure of basis matrix] \label{thm:vij_tensor}
Let $Q_i$, $i=1,\ldots,k$  be the sequence of
basis matrices generated during
the execution of 
 $k-1$ iterations of Algorithm~\ref{alg:infarnoldi}.
Then, all block elements of the basis matrix $Q_k$
(when partitioned as \eqref{eq:Vkblocks}) 
are linear combinations of $q_{1,1},\ldots,q_{1,k}$.
\end{lem}
\begin{proof}
The proof is based on induction over the iteration count $k$.
The result is trivial for $k=1$.
Suppose the results holds for some $k$.
Due to the fact that $Q_k$ is the leading submatrix of $Q_{k+1}$, as in \eqref{eq:Vkblocks},
we only need to show that the blocks of the
new column are a linear combinations 
$q_{i,j}$, $i, j=1,\ldots,k$. 
This follows
directly from the fact that  $q_{2,k+1},\ldots,q_{k+1,k+1}$
is (in step 3-9 in Algorithm~\ref{alg:infarnoldi})
constructed as linear combination
of $q_{1,k},\ldots,q_{k,k}$. See Figure~\ref{fig:comp_tree}
\end{proof}

 We note that
the structure presented in Lemma~\ref{thm:vij_tensor} 
is very natural in view of similar structures
in other settings \cite[Section~3.1]{Kressner:2014:TOAR},
\cite[Page 1057]{Zhang:2013:MOR} and \cite[Theorem~4.4]{VanBeeumen:2014:CORK}.

\subsection{Derivation of TIAR}

We now know from Lemma~\ref{thm:vij_tensor} 
that the basis matrix in IAR has a redundant structure.
In this section we show 
that this structure can be exploited such that Algorithm~\ref{alg:infarnoldi} can be equivalently 
reformulated as an iteration involving a tensor factorization of the basis matrix without redundancy. 
%Similar structures in the basis matrix
%have been observed
%and exploited in other settings in other works.
% The structure is present for Arnoldi's
%method applied to companion linearizations \cite{Kressner:2014:TOAR}, and can be extended to rational Krylov methods 
%\cite{VanBeeumen:2014:CORK}, where instead of a tensor
%factorization, the algorithms are interpreted as having
% two-levels of orthogonalizations.
%A similar
%redundancy structure is 
%exploited in the adaption of the IAR (Chebyshev version)
%in \cite[Page 1057]{Zhang:2013:MOR} where
%it used to carry out model order reduction.
%
We present a different formulation 
involving a factorization with a tensor which allows
us to improve IAR both 
in terms of memory and computation time.
This equivalent, but improved, version of
Algorithm~\ref{alg:infarnoldi}
appears to be competitive in general, and
can be considerably specialized to the waveguide eigenvalue
problem as we show in section~\ref{sect:adaption}.

More precisely, Lemma~\ref{thm:vij_tensor}
implies that there exists  
$a_{i,j,\ell}$ for $i,j,\ell=1,\ldots,k$ such that 
\begin{equation}
   q_{i,j}=\sum_{\ell=1}^ka_{i,j,\ell} z_\ell,
\textrm{ for }\,i,j=1,\ldots,k
\label{eq:vij_tensor}
\end{equation}
where $z_1,\ldots,z_k$ is a basis of 
the span of the $k$ first columns of the first block row, i.e., 
$\operatorname{span}(q_{1,1},\ldots,q_{1,k})=\operatorname{span}(z_{1},\ldots,z_{k})$.
Due to \eqref{eq:vij_tensor}, 
the quantities $[z_1,\ldots,z_k]$, 
$a_{i,j,\ell}$ for $i,j,\ell=1,\ldots,k$ can be 
interpreted
as a factorization of $Q_k$.
For reasons of numerical stability 
we 
here work with an orthonormal basis $z_1,\ldots,z_k$,
i.e., $Z_k:=[z_1,\ldots,z_k]\in\CC^{n\times k}$ is an 
orthogonal matrix.
This is not
a restriction if the 
columns of the first block
row of $Q_{k}$ are linearly independent. 
Note that the first block row of $Q_k$
can only be linearly independent if 
 $k\le n$. This
is the case for large-scale
nonlinear eigenvalue problems, 
as the one we consider in this paper.

Suppose for the moment  that we have
carried out $k-1$ iterations of Algorithm~\ref{alg:infarnoldi}.
From Lemma~\ref{thm:vij_tensor}
we know that the basis matrix can be factorized according
to \eqref{eq:vij_tensor}. 
The following results show that one loop, i.e., steps 3-11, can be carried
out without explicitly storing $Q_k$, but instead 
only storing the factorization \eqref{eq:vij_tensor}
represented by the tensor $a_{i,j,\ell}$ for  $i,j,\ell=1,\ldots,k$
 and the matrix $Z_k\in\CC^{n\times k}$.
Instead of carrying out operations on $Q_k$ 
that lead to $Q_{k+1}$, we construct
equivalent operations on the factorization of
$Q_k$, i.e., 
$a_{i,j,\ell}$ for  $i,j,\ell=1,\ldots,k$
 and the matrix $Z_k\in\CC^{n\times k}$,
that directly lead to the factorization of $Q_{k+1}$,
i.e.,  $a_{i,j,\ell}$ for  $i,j,\ell=1,\ldots,k+1$
 and the matrix $Z_{k+1}\in\CC^{n\times (k+1)}$,
without explicitly forming $Q_k$ or $Q_{k+1}$.

To this end, suppose we have $a_{i,j,\ell}$ for $i,j,\ell=1,\ldots,k$
 and $z_{1},\ldots,z_{k}$  available after $k-1$ iterations
such that \eqref{eq:vij_tensor} is satisfied, and
consider the steps 3-11 one-by-one.
In Step 3 we need to compute the vectors $y_2,\ldots,y_{k+1}$. 
They can be computed  from the factorization of $Q_k$, 
since 
\begin{equation}  \label{eq:yj_z}
  y_j=\frac{1}{j-1}q_{j-1,k}=\frac{1}{j-1}\sum_{\ell=1}^ka_{j-1,k,\ell} z_\ell,
\end{equation}
for $j=2,\ldots,k+1$. The vector $y_1$ is (in Step 4)
computed using  \eqref{eq:y0} and $y_2,\ldots,y_{k+1}$ 
and does not explicitly require the basis matrix. For
 reasons of efficiency (which we further discuss in Remark~\ref{rem:complexity})
% \mytodo{Explain why must be so better, maybe with a reference...see numerical simulations}
 we carry out \eqref{eq:yj_z} with an equivalent matrix-matrix multiplication, 
\begin{equation}  \label{eq:yj_zmat}
  \begin{bmatrix}\tilde{y}_2 & \cdots& \tilde{y}_{k+1}
\end{bmatrix}=
  Z_k A_k,
\end{equation}
where $A_k^T=[a_{i,k,\ell}]_{i,j=1}^k$
and subsequently setting 
\begin{equation} \label{eq:yj_tilde}
y_j=\frac{1}{j-1}\tilde{y}_j\textrm{ for }j=2,\ldots,k+1.  
\end{equation}

In order to efficiently carry out the Gram-Schmidt orthogonalization
process in step 6-9, it turns out to be
efficient to first form a new vector $z_{k+1}$, which can be used in the factorized representation of $Q_{k+1}$. 
We \emph{define} a new 
vector $z_{k+1}$ via a Gram-Schmidt orthogonalization
 of $y_1$ against $z_1,\ldots,z_k$. 
That is, we compute 
$z_{k+1}\in\CC^n$ and $t_1,\ldots,t_{k+1}\in\CC$  such that 
\begin{equation} % \label{eq:}
\label{eq:y0_gm}
  y_1=t_1z_1+\cdots+t_{k+1}z_{k+1}
\end{equation}
and expand $Z_{k+1}:=\begin{bmatrix}Z_k & z_{k+1}\end{bmatrix}$ such that $Z_{k+1}^HZ_{k+1}=I$. 

The new vector  $y$ (formed in Step 5) can now be expressed 
using the factorization, since 
\begin{equation}\label{eq:y_kron}
  y=
  \begin{bmatrix}
   y_1\\
   \frac11q_{1,k}\\
   \frac12q_{2,k}\\
   \vdots\\
   \frac1kq_{k,k}\\ 
  \end{bmatrix}=
e_1\otimes y_1+
  \sum_{\ell=1}^k\begin{bmatrix}
   0\\
   \frac11a_{1,k,\ell}\\
   \frac12a_{2,k,\ell}\\
   \vdots\\
   \frac1ka_{k,k,\ell}  
  \end{bmatrix}\otimes z_\ell
= 
  \sum_{\ell=1}^{k+1}\begin{bmatrix}
   g_{1,\ell}\\
   \vdots\\
    g_{k+1,\ell}\\
  \end{bmatrix}\otimes z_\ell
\end{equation}
where we have defined $g_{i,\ell}$ as 
\begin{subequations}\label{eq:giell_def}
\begin{eqnarray}
  g_{1,\ell}&:=&t_\ell \textrm{ for } \ell=1,\ldots,k+1\\
  g_{i,\ell}&:=&\frac{1}{i-1}a_{i-1,k,\ell} 
\textrm{ for } 
i=2,\ldots,k+1,\;\;
\ell=1,\ldots,k,\\
  g_{i,k+1}&:=&0 \textrm{ for }i=2,\ldots,k+1.
\end{eqnarray}
\end{subequations}
Instead of explicitly working with $y$, we 
store the 
matrix $[g_{i,\ell}]_{i,\ell=1}^{k+1}\in\CC^{(k+1)\times (k+1)}$, 
representing the blocks of $y$ as linear combinations of $Z_{k+1}$. 

In order to derive a procedure to compute
$h\in\CC^k$ (in Step 6) without explicitly using $Q_k$, 
it is convenient to express the
relation \eqref{eq:vij_tensor} using Kronecker products. 
We have
\begin{equation}
  Q_k=
\sum_{\ell=1}^{k}
\begin{bmatrix}
  a_{1,1,\ell}	& \cdots 	& a_{1,k,\ell}	\\
  \vdots 	&  		& \vdots	\\
  a_{k,1,\ell}	& \cdots 	& a_{k,k,\ell} \\
\end{bmatrix}\otimes z_\ell.\label{eq:Vk_kron}
\end{equation}

%the computation of $h$ by using the
%orthogonality of $Z_{k+1}$. 
From the definition of $h$ and 
\eqref{eq:Vk_kron} combined with \eqref{eq:y_kron}
and the orthogonality of $Z_{k+1}$,
we can now see that $h$ can be expressed without explicitly 
using $Q_k$ as follows
\begin{align}\label{eq:h_ag}
  h &=
\underline{Q}_k^Hy=
\left(\sum_{\ell=1}^{k}
\begin{bmatrix}
  a_{1,1,\ell}^*& \cdots & a_{k,1,\ell}^* & 0\\
  \vdots &  & \vdots & \vdots\\
  a_{1,k,\ell}^*& \cdots & a_{k,k,\ell}^* & 0
\end{bmatrix}\otimes z_\ell^H\right)
\left(
  \sum_{\ell'=1}^{k+1}\begin{bmatrix}
   g_{1,\ell'}\\
   \vdots\\
    g_{k+1,\ell'}\\
  \end{bmatrix}\otimes z_{\ell'}
\right)
\\ & =
\sum_{\ell=1}^{k}
\begin{bmatrix}
  a_{1,1,\ell}^*& \cdots & a_{k,1,\ell}^*\\
  \vdots &  & \vdots\\
  a_{1,k,\ell}^*& \cdots & a_{k,k,\ell}^*
\end{bmatrix}\begin{bmatrix}
   g_{1,\ell}\\
   \vdots\\
    g_{k,\ell}\\
  \end{bmatrix}.	\nonumber
\end{align}
In Step 7 we need 
to compute the orthogonal completement of $y$ 
with respect to $\underline{Q}_k$. This can be
represented without explicit use of $Q_k$ 
as follows
\begin{align}\label{eq:yperp_f}
y_\perp & =y-\underline{Q}_kh
=
\sum_{\ell=1}^{k+1}\begin{bmatrix}
   g_{1,\ell}\\
   \vdots\\
    g_{k+1,\ell}\\
  \end{bmatrix}\otimes z_\ell
-
\sum_{\ell=1}^{k}
\left(\begin{bmatrix}
  a_{1,1,\ell}& \cdots & a_{1,k,\ell}\\
  \vdots &  & \vdots\\
  a_{k,1,\ell}& \cdots & a_{k,k,\ell}\\
   0 & \cdots & 0
\end{bmatrix}h\right)\otimes z_\ell\\ & =
\sum_{\ell=1}^{k+1}\begin{bmatrix}
   f_{1,\ell}\\
   \vdots\\
    f_{k+1,\ell}\\
  \end{bmatrix}\otimes z_\ell,	\nonumber
\end{align}
where we have used the elements
of the matrix $[f_{i,j}]_{i,\ell=1}^{k+1}$
with columns defined by
\begin{equation}  \label{eq:fdef}
\begin{bmatrix}
  f_{1,\ell}\\
  \vdots\\
  f_{k+1,\ell}
\end{bmatrix}:=
\begin{bmatrix}  
g_{1,\ell}\\
\vdots\\
g_{k+1,\ell}
\end{bmatrix}
-
\begin{bmatrix}
  a_{1,1,\ell}& \cdots & a_{1,k,\ell}\\
  \vdots &  & \vdots\\
  a_{k,1,\ell}& \cdots & a_{k,k,\ell}\\
   0 & \cdots & 0
\end{bmatrix}h
\end{equation}
for $\ell=1,\ldots,k$ and $f_{i,k+1}:=g_{i,k+1}$ for  $i=1,\ldots,k+1$. 
%\[
%f_{:,\ell}:=g_{:,\ell}-a_{:,:,\ell}h.
%\]

We need $\beta$ in Step~8, which is defined as the Euclidean norm of $y_\perp$.
Due to the  orthogonality of $Z_{k+1}$, we can 
also express $\beta$ without using vectors of length $n$. 
In fact, it turns out that $\beta$
is the Frobenius norm of the matrix $[f_{i,j}]_{i,j=1}^{k+1}$,
since
\begin{align*}
  \beta^2&=\|y_\perp\|^2 =
\left(\sum_{\ell=1}^{k+1}\begin{bmatrix}
   f_{1,\ell}\\
   \vdots\\
    f_{k+1,\ell}
  \end{bmatrix}\otimes z_\ell\right)^H
\left(
\sum_{\ell'=1}^{k+1}\begin{bmatrix}
   f_{1,\ell'}\\
   \vdots\\
    f_{k+1,\ell'}
  \end{bmatrix}\otimes z_{\ell'},
\right)\\
&=\sum_{\ell=1}^{k+1}
\begin{bmatrix}
   f_{1,\ell}\\
   \vdots\\
    f_{k+1,\ell}
  \end{bmatrix}^H\begin{bmatrix}
   f_{1,\ell}\\
   \vdots\\
    f_{k+1,\ell}
  \end{bmatrix}=\left\|[f_{i,j}]_{i,j=1}^{k+1}\right\|_{\rm frob}^2.
\end{align*}

Finally (in Step 11), we
expand $Q_k$ by one column corresponding 
$q_{k+1}$, which is the  normalized orthogonal complement. 
By using the introduced matrix $[f_{i,j}]_{i,j=1}^{k+1}$
we have that
\begin{equation} \label{eq:vkp1}
  q_{k+1}=\frac{1}{\beta}y_\perp=
\frac{1}{\beta}\sum_{\ell=1}^{k+1}\begin{bmatrix}
   f_{1,\ell}\\
   \vdots\\
    f_{k+1,\ell}\\
  \end{bmatrix}\otimes z_\ell.  
\end{equation}

Let us now define
\begin{subequations}\label{eq:akp1ell}
\begin{eqnarray}
  a_{i,k+1,\ell}&=&\frac{1}{\beta}f_{i,\ell},
\textrm{ for }i,\ell=1,\ldots,k+1,\\%, \ell=1,\ldots,k+1,  \\
  a_{k+1,j,\ell}&=&0\textrm{ for } \ell=1,\ldots,k+1, j=1,\ldots,k\\
  a_{i,j,k+1}&=&0\textrm{ for } i=1,\ldots,k+1, j=1,\ldots,k.
\end{eqnarray}
\end{subequations}
Hence, $a_{i,j,\ell}$ for $i,j,\ell=1,\ldots,k+1$
and $Z_{k+1}$ can be seen as 
a factorization of $Q_{k+1}$
in the sense of \eqref{eq:vij_tensor}, 
since the column 
added in comparison to the
factorization of $Q_k$ is precisely \eqref{eq:vkp1}.

We summarize the above reasoning
with a precise result
showing how
the dependence on $Q_k$ 
for every step in Algorithm~\ref{alg:infarnoldi}
can be removed, including how a factorization
of $Q_{k+1}$ can be constructed.

\begin{thm}[Equivalent steps of algorithm] \label{thm:equiv}
Let $Q_k$ be
the basis matrix generated by $k-1$ iterations of 
Algorithm~\ref{alg:infarnoldi} and
suppose $a_{i,j,\ell}$, for $i,j,\ell=1,\ldots,k$
 and $Z_k$ are given 
such that they represent a factorization of $Q_k$
of the type \eqref{eq:vij_tensor}.
The quantities computed (by executing Steps~3-11) 
in iteration $k$ satisfy the following relations.
\begin{itemize}
  \item[(i)] The vectors $y_2,\ldots,y_{k+1}$ computed in {\bf Step 3}, satisfy \eqref{eq:yj_zmat}.
  \item[(ii)] Suppose $y_1$ (computed in Step 3)  satisfies
$y_1\not\in\operatorname{span}(z_1,\ldots,z_{k})$. 
Let $z_{k+1}\in\CC^n$ and $t_1,\ldots,t_{k+1}\in\CC$ 
be the result of the Gram-Schmidt process 
satisfying \eqref{eq:y0_gm}. Moreover, let $[g_{i,\ell}]_{i,\ell=1}^{k+1}$ be defined by
\eqref{eq:giell_def}.
  Then, then $h$ computed in {\bf Step~6}, satisfies 
\eqref{eq:h_ag}.
  \item[(iii)] Let $[f_{i,\ell}]_{i,\ell=1}^{k+1}$
be defined by \eqref{eq:fdef}.
Then, the vector $y_\perp$, computed in {\bf Step~8},  satisfies 
\eqref{eq:yperp_f}.
  \item[(iv)] The scalar $\beta$, computed in {\bf Step~9}, 
satisfies $\beta=\left\|[f_{i,\ell}]_{i,\ell=1}^{k+1}\right\|_{\rm fro}$
\end{itemize}
Moreover, if we expand  $a_{i,j,\ell}$ as in \eqref{eq:akp1ell}, then, $a_{i,j,\ell}$, for $i,j,\ell=1,\ldots,k+1$ and $z_1,\ldots,z_{k+1}$
represent a factorization of $Q_{k+1}$ in the sense
that \eqref{eq:vij_tensor}
is satisfied for $k+1$.
\end{thm}

The above theorem directly gives
us a practical algorithm. We state it
explicitly in Algorithm~\ref{alg:infarnoldi2}.
The details of
the (possibly) repeated Gram-Schmidt process in Step~9
is straightforward and left out for brevity.
%consists of replacing $G$ in Step 6 by
%$F$ and summing the old and new $h$ and $F$. 
\begin{algorithm}%\SetLine
\caption{Tensor infinite Arnoldi method - TIAR  \label{alg:infarnoldi2}}%
\SetKwInOut{Input}{Input}\SetKwInOut{Output}{output}
\Input{$x_1\in\CC^n$}
\BlankLine
\nl Let $Q_1=x_1/\|x_1\|_2$,  $H_0=$empty matrix\\
\For{$k=1,2,\ldots,m$}{
\nl Compute $y_2,\ldots,y_{k+1}$ from  $a_{i,k,\ell}$,
$i,\ell=1,\ldots,k$  and $Z_k$ using 
\eqref{eq:yj_zmat}-\eqref{eq:yj_tilde}
\label{step:infarn2:y2yk}\\
%%%€
%%%€ from the last column of $V_k$
%%%€by setting 
%%%€\[
%%%€  y_j=\frac{1}{j}\sum_{\ell=1}^ka_{j,k,\ell} z_\ell
%%%€\]
%%%€for $j=2,\ldots,k+1$.
\nl Compute $y_1$ from $y_2,\ldots,y_{k+1}$ using \eqref{eq:y0}\label{step:infarn2:y1}\\
%%%\STATE Let $w_k:=\vect(y_1,\ldots,y_{k+1})$.
\nl Compute $t_1,\ldots,t_{k+1}$ and $z_{k+1}$ 
by orthogonalizing $y_1$ against $z_1,\ldots,z_k$ 
using a (possibly repeated) Gram-Schmidt process 
such that \eqref{eq:y0_gm} is satisfied.\label{step:infarn2:Zorth}\\
\nl Compute the matrix $G=[g_{i,\ell}]_{i,\ell=1}^{k+1}$ using \eqref{eq:giell_def}\\
\nl Compute $h\in\CC^k$ using \eqref{eq:h_ag}\label{step:infarn2:h}\\
%%%\STATE Compute $a_{i,k+1,\ell}$ for  $i=2,\ldots,k+1$, $\ell=1,\ldots,k$
%%% using \eqref{eq:akp1def}
%%%\STATE Compute $v_{1,k+1}=...$ 
\nl Compute the matrix $F=[f_{i,\ell}]_{i,\ell=1}^{k+1}$ using \eqref{eq:fdef}\label{step:infarn2:F}\\
\nl Possibly repeat Steps~\ref{step:infarn2:h}-\ref{step:infarn2:F} and obtain updated  $h$ and  $F$\\
\nl Compute $\beta=\left\|F\right\|_{\rm fro}$ \label{step:infarn2:beta}\\
\nl Expand $a_{i,j,\ell}$ using \eqref{eq:akp1ell}\\
%%%$a_{1,k+1,\ell}=\frac{1}{\beta}f_\ell$, $\ell=1,\ldots,k+1$ and $Z_{k+1}=(Z_k,z_{k+1})$.
%%%\STATE $\hat{w}_k = w_k - V_k h_k$ (optimize)
%%%\STATE Compute $\beta_k = \|\hat{w}_k\|_2$ and let $v_{k+1} = \hat{w}_k / \beta_k$
%%%
%%%$w_k=\vect(y_0,\ldots, y_{k+1})$ according to
%%%Remark~\ref{rem:scalarproduct} by computing 
%%%$h_k\in\CC$, $\beta_k$, $\hat{w}_k\in\CC^{n(k+2)}$.
%%%\STATE Let $w_k:=\vect(y_0,\ldots,y_{N+1})$, compute $h_k=V_k^* w_k$
%%%and then $\hat{w}_k = w_k - V_k h_k$
%%%\STATE Compute $\beta_k = \|\hat{w}_k\|_2$ and let $v_{k+1} = \hat{w}_k / \beta_k$
\nl Let $H_k = \left[\begin{array}{cc}H_{k-1} & h \\ 0 & \beta \end{array}\right] \in\CC^{(k+1)\times k}$\\
%%%\STATE Expand $V_k$ into $V_{k+1} = [V_k,v_{k+1}]$
}
\nl Compute the eigenvalues $\{\mu_i\}_{i=1}^m$ of the leading $m\times m$ submatrix
of the Hessenberg matrix $H_k$ 	\\
\nl Return approximations $\{1/\mu_i\}_{i=1}^m$
\end{algorithm}

\begin{remark}[Computational performance of IAR and TIAR]\label{rem:complexity}\rm
Under the condition that $q_{1,1},\ldots, q_{1,m}$ 
are linearly independent, 
Algorithm~\ref{alg:infarnoldi} (IAR) and 
Algorithm~\ref{alg:infarnoldi2} (TIAR) are equivalent
in exact arithmetic. 
The required computational resources 
of the two algorithms are however very different and 
TIAR appears to  be preferable over IAR, in general.

% MEMORY REQUIREMENT
The first advantage of TIAR concerns the memory requirements. 
More precisely, in TIAR, the basis matrix is stored using 
a tensor $[a_{i,j,\ell}]_{i,j,\ell=1}^m \in \mathbb{C}^{m \times m \times m}$ and a matrix 
 $Z_m \in \mathbb{C}^{n \times m}$. Therefore, TIAR requires the storage of  $\mathcal{O}(m^3)+\mathcal{O}(mn)$ numbers.
In contrast to this, in IAR we need to store   $\mathcal{O}(m^2n)$ numbers
since the basis matrix $Q_m$ is of size $mn \times m$.
Therefore, assuming that $m \ll n$, TIAR requires much less memory than IAR.

% the basis matrix $Q_m \in \mathbb{C}^{m n \times m}$, 
% whereas in TIAR we represent the basis matrix though 
% a tensor $a_{m,m,m} \in \mathbb{C}^{m \times m \times m}$ and a matrix 
%  $Z_m \in \mathbb{C}^{n \times m}$. Therefore is required the storage of 
%  $\mathcal{O}(m^3)+\mathcal{O}(mn)$ numbers. 

% In IAR is needed to 
% store the matrix $Q_m \in \mathbb{C}^{m n \times m}$ that require the the storage 
% of $\mathcal{O}(m^2n)$ numbers. Instead in TIAR is needed to store a tensor
% $a_{m,m,m} \in \mathbb{C}^{m \times m \times m}$ and a matrix 
% $Z_m \in \mathbb{C}^{n \times m}$, hence is required the storage of 
% $\mathcal{O}(m^3)+\mathcal{O}(mn)$ numbers, that assuming $n$ big and $m<<n$ is much less 
% than the memory requirement of IAR.

% COMPLEXITY
The essential computational effort of carrying out $m$ steps of IAR 
consists of: $m$ linear solves, computing
$\sum_{i=1}^kM^{(i)}(0)x_i$, for $k=1,\ldots,m$, 
and orthogonalizing a vector of length $kn$ 
against $k$ vectors of size $kn$ for $k=1,\ldots,m$. The orthogonalization
has complexity
\begin{equation}\label{eq:t_IAR_orth}
 t_{\rm IAR,orth}(m,n)=\mathcal{O}(m^3n),
\end{equation}
which is the dominating cost when the linear solves are relatively cheap 
as in the waveguide eigenvalue problem.

On the other hand, the computationally dominating part of carrying out
$m$ steps of TIAR is as follows. 
Identical to IAR, $m$ steps require $m$ linear solves, and 
the computation of  $\sum_{i=1}^kM^{(i)}(0)x_i$, for $k=1,\ldots,m$. 
The orthogonalization process in TIAR (Step \ref{step:infarn2:Zorth}-\ref{step:infarn2:beta}) is computationally cheaper than IAR. More precisely, 
\[
 t_{\rm TIAR,orth}(m,n)=\mathcal{O}(m^2n).
\]
Unlike IAR, TIAR requires a computational effort 
in order to access the vectors $y_2,\ldots,y_k$ in Step~\ref{step:infarn2:y2yk}
since they are implicitly given via $a_{i,j,k}$ and $Z_k$. 
In Step~\ref{step:infarn2:y2yk} we compute 
$y_2,\ldots,y_k$  with \eqref{eq:yj_tilde}
and \eqref{eq:yj_zmat} which correspond to  multiplying
a matrix of size $n\times k$ with a matrix of size $k\times k$ (and subsequently scaling the vectors). Hence, the operations corresponding to Step~\ref{step:infarn2:y2yk} for  $m$ iterations of TIAR can be carried out in
\begin{equation}\label{eq:t_TIARStep2}
 t_{\rm TIAR, Step~\ref{step:infarn2:y2yk}}(m,n) =\mathcal{O}(m^3n).
\end{equation}

At first sight, nothing is gained since the complexity of the
orthogonalization in IAR \eqref{eq:t_IAR_orth}
and Step~2 of 
TIAR, are both $\mathcal{O}(m^3n)$. However, it turns out 
that TIAR is often considerably faster in practice. 
This can be explained as follows. 
In the orthogonalization process of IAR we must compute 
$h=\underline{Q}_k^Ty$ where $\underline{Q}_k\in\CC^{kn\times k}$, whereas  in Step~2 in TIAR we must compute
$Z_kA_k^T$ (in \eqref{eq:yj_zmat}) where $Z_k\in\CC^{n\times k}$ 
and $A_k\in\CC^{k\times k}$. Note that the operation $Z_kA_k^T$
 involves $nk+k^2$ values, whereas $\underline{Q}_k^Ty$
involves $nk^2$ values, i.e., Step~2 in TIAR 
involves less data.
This implies that on modern computer architectures, 
where CPU caching makes operations on small 
data-sets more efficient, it is in practice
 considerably faster to compute $Z_kA_k^T$ than 
 $\underline{Q}_k^Ty$ 
although the operations have the same computational complexity. 
This difference is also verified in the simulations in 
section~\ref{sect:simulations}.
\end{remark}

\section{Adaption to the waveguide problem}\label{sect:adaption}

\subsection{Cayley transformation}\label{sect:cayley}

One interpretation of IAR
involves a derivation via the truncated Taylor
series expansion.
The truncated Taylor expansion is expected
to converge slowly for points close to 
branch-point singularies, 
and in general not converge at all for
points further away from the origin
than the closest singularity. 
%, which only converges within
%the convergence disk of the analytic function.
Note that $M$ defined
in \eqref{eq:nep} has branch point
singularities at the roots of $\beta_{\pm,k}(\gamma)$, $k=-p,\ldots,p$ where $\beta_{\pm,k}$ is defined in \eqref{eq:betadef}. 
In our situation, the eigenvalues of interest are
close to the imaginary axis and, since the roots of $\beta_{\pm,k}$ are
purely imaginary, the singularities
are purely imaginary, which suggests
poor performance of IAR (as well as TIAR) when applied to $M$.

In order to resolve this, we  
first carry out a Cayley transformation 
which moves the singularities 
to the unit circle and the eigenvalues of interest
to points inside the unit disk, 
%such that 
%the eigenvalues
%of interest are always
i.e., inside the convergence disk.

In our setting, the Cayley transformation 
for a shift $\gamma_0\in (-\infty, 0) \times (-2\pi, 0)$ is given by
\begin{equation}  \label{eq:cayley}
  \lambda=\frac{\gamma-\gamma_0}{\gamma+\conj{\gamma}_0},
\end{equation}
and its inverse is 
\begin{equation}  \label{eq:cayleyinv}
  \gamma=\frac{\gamma_0+\lambda\conj{\gamma}_0}{1-\lambda}.
\end{equation}
The shift should
be chosen close the eigenvalues of interest, 
i.e., close to the imaginary axis.
However, the shift should not be chosen too close to the imaginary axis,
because this generates a transformed problem where the 
eigenvalues are close to the singularities.

Note that the transformed problem is still
a nonlinear eigenvalue problem of the type
\eqref{eq:nep}, and we can easily remove the poles introduced
by the denominator in \eqref{eq:cayley}. More precisely,
we work with the transformed
 nonlinear eigenvalue problem
\begin{eqnarray}  
 \tilde{M}(\lambda)&:=&
 \begin{bmatrix}
   (1-\lambda)^2I &  \\
   &   (1-\lambda)I
 \end{bmatrix}
M(\tfrac{\gamma_0+\lambda\conj{\gamma}_0}{1-\lambda})\notag
\\\label{eq:Mtildedef}
&=&\begin{bmatrix}
F_A(\lambda) & 
%F(\lambda,A_0,A_1,A_2) & 
%F(\lambda,C_{1,0},C_{1,1},C_{1,2})\\
F_{C_1}(\lambda)\\
  (1-\lambda)C_2^T & \tilde{P}(\lambda)
\end{bmatrix},
\end{eqnarray}
where 
%\begin{subequations}
\begin{eqnarray}
%\end{subequations}
F_A(\lambda)&:=&
(1-\lambda)^2A_0+ (\gamma_0+\lambda\conj{\gamma}_0)(1-\lambda)A_1+	(\gamma_0+\lambda\conj{\gamma}_0)^2A_2, \notag\\
F_{C_1}(\lambda)&:=&
(1-\lambda)^2C_{1,0}+ (\gamma_0+\lambda\conj{\gamma}_0)(1-\lambda)C_{1,1}+	(\gamma_0+\lambda\conj{\gamma}_0)^2C_{1,2}, \notag\\
\tilde{P}(\lambda)&:=&(1-\lambda)P(\frac{\gamma_0+\lambda\conj{\gamma}_0}{1-\lambda}).\label{eq:ptildeeq}
\end{eqnarray}

%$F_B(\lambda,B_0,B_1,B_2):=
%(1-\lambda)^2B_0+ (\gamma_0+\lambda\conj{\gamma}_0)(1-\lambda)B_1+	(\gamma_0+\lambda\conj{\gamma}_0)^2B_2$.

\subsection{Efficient computation of $y_1$} \label{sect:iar_adaption}
In order apply IAR or TIAR to the waveguide
problem, we need to provide a procedure to compute
$y_1$ in Step~\ref{step:infarn2:y1} of
Algorithm~\ref{alg:infarnoldi} and 
Algorithm~\ref{alg:infarnoldi2} 
using \eqref{eq:y0}. 
The structure of $\tilde{M}$ in \eqref{eq:Mtildedef}
can be explicitly exploited and merged with Step~2 as follows.
We analyze \eqref{eq:y0} for $k\ge 3$. It is straightforward 
to compute the corresponding formulas for $k<3$.
Due to the definition of $\tilde{M}$ in
\eqref{eq:Mtildedef}, formula \eqref{eq:y0} can be expressed
as
\begin{equation}\label{eq:y1z}
-\tilde{M}(0)y_1=z_1+z_2,
\end{equation}
with% $F'_{A}=F_\lambda(0,A_0,A_1,A_2)$
%\begin{subequations} 
\begin{eqnarray}
  z_1&:=&
  \begin{bmatrix}
    F'_A(0)y_{2,1}+F'_{C_1}(0)y_{2,2}  +
    F''_A(0)y_{3,1}+F''_{C_1}(0)y_{3,2}  
\\
-C_2^Ty_{2,1}
  \end{bmatrix},  \notag\\
  z_2&:=&
  \begin{bmatrix}
    0\\
    \sum_{i=1}^k\tilde{P}^{(i)}(0)y_{i+1,2}.
  \end{bmatrix}\label{eq:z2_structure}
\end{eqnarray}
%\end{subequations}
where we have decomposed $y_i^T=(y_{i,1}^T,y_{i,2}^T)$,
$i=1,\ldots,k$  
with
$y_{i,1}\in\CC^{n_xn_z}$ and
$y_{i,2}\in\CC^{2n_z}$. The linear
system of equations \eqref{eq:y1z}
can be solved by precomputing the Schur complement and its 
LU-factorization, which is not a dominating
component (in terms of execution time) in our situation.
The vector $z_1$ can be computed directly by using the definition of $F_A$ and $F_{C_1}$. 
Using the definition of $\tilde{P}$ in \eqref{eq:ptildeeq}, 
we can express the bottom block of $z_2$ as
\begin{equation} % \label{eq:}
  z_{2,2}=
   \begin{bmatrix}
R &0
  \\
0 & R
\end{bmatrix}
\sum_{i=1}^k 
D_i
\left(
   \begin{bmatrix}
R^{-1} &0
  \\
0 & R^{-1}
\end{bmatrix}
y_{i+1,2}
\right)\in\CC^{2n_z},
\end{equation}
where  $D_i:=\diag(\alpha_{-,-p,i},\ldots,\alpha_{-,p,i},
\alpha_{+,-p,i},\ldots,
\alpha_{+,p,i})$
with 
\begin{equation} \label{eq:coeff_DtN}
  \alpha_{\pm,j,i}
:=\left(\frac{d^i}{d\lambda^i}
\left((1-\lambda)(s_{\pm,j}(
\frac{\gamma_0+\lambda\conj{\gamma}_0}{1-\lambda}) +d_0)%
\right)\right)_{\lambda=0}.
\end{equation}
In order to carry out $m$ steps of the algorithm,
we need to evaluate \eqref{eq:coeff_DtN}
 $2n_zm$ times.
%which involves derivatives of square roots and polynomials. 
We propose to do this with the
efficient recursion formula given Appendix~\ref{sect:alphas}.
We note that similar formulas 
are used in \cite{Tausch:2000:WAVEGUIDE} for slightly
different functions.

Although the above formulas can be used directly
to compute $y_1$, further performance
improvement can be achieved by 
considerations of Step~2. Note that the complexity of
Step~2 in TIAR is $\mathcal{O}(m^3n)$, as given in equation
\eqref{eq:t_TIARStep2}. The computational complexity
of this step can be decreased by using the
fact that in order to compute $y_1$ in Step~3 and equation \eqref{eq:y1z}-\eqref{eq:z2_structure}, 
we only need $y_2,y_3$ and $y_{4,2},\ldots,y_{k+1,2}\in\CC^{2n_z}$, i.e., not the full vectors. 
The structure can exploited in the operations in Step~2 
as follows.

Let $B_{11}\in\CC^{2\times 2},B_{12}\in\CC^{2\times (k-2)},B_{21}\in\CC^{(k-2)\times 2}$ and $B_{22}\in\CC^{(k-2)\times(k-2)}$ be defined as blocks
of $A_k$, 
\begin{equation}\label{eq:Akdecomposition}
 A_k=
 \begin{bmatrix}
   B_{11} & B_{12}\\
   B_{21} & B_{22}
 \end{bmatrix}.
\end{equation}
From  \eqref{eq:yj_zmat} and \eqref{eq:yj_tilde} we have 
\begin{equation}\label{eq:y2y3_expl}
  \begin{bmatrix}\tilde{y}_2 & \tilde{y}_3\end{bmatrix}=Z_k
\begin{bmatrix}
B_{11}\\
B_{21}
\end{bmatrix},\;\; 
\begin{bmatrix}y_2 & y_3\end{bmatrix}=
\begin{bmatrix}\tilde{y}_2 & \tilde{y}_3/2 \end{bmatrix}, 
\end{equation}
and
\begin{equation}\label{eq:ytilde_expl}
 \begin{bmatrix}\tilde{y}_{4,2}\; \ldots\; \tilde{y}_{k+1,2}\end{bmatrix}=Z_{k,2}B_{22},\;\;
\begin{bmatrix}y_{4,2}& \ldots& y_{k+1,2}\end{bmatrix}=
\begin{bmatrix}\tilde{y}_{4,2}/3 & \ldots& \tilde{y}_{k+1,2}/k\end{bmatrix},
\end{equation}
where $Z_{k,2}\in\CC^{2n_z\times (k-2)}$ consists of the trailing  block of $Z_k$. 

%The above reasoning  a
%procedure to compute $y_1$ by merging
%With the above reasoning we can merge
By using formulas \eqref{eq:y2y3_expl}-\eqref{eq:ytilde_expl},
we merge Step~2 and Step~3 in Algorithm~\ref{alg:infarnoldi2}
such that we can compute $y_1$ 
without computing the full vectors $y_2,\ldots,y_k$. 
For future reference we call this adaption WTIAR.

As explained in Remark~\ref{rem:complexity},
Step~2 of TIAR
is the dominating component 
in terms of asymptotic complexity.
With the adaption explained in 
\eqref{eq:Akdecomposition}-\eqref{eq:ytilde_expl},
the complexity of Step~2 in WTIAR
is 
\begin{equation}\label{eq:step2_WTIAR}
  t_{\rm WTIAR,Step~2}(m,n)=\mathcal{O}(nm^2)+\mathcal{O}(n_zm^3).
\end{equation}
If  the problem is discretized with the same
number of discretization points in $x$-direction and $z$-direction, 
we have 
  $t_{\rm WTIAR,Step~2}(m,n)=\mathcal{O}(nm^2)+\mathcal{O}(\sqrt{n}m^3)$,
which is considerable better than the complexity \eqref{eq:t_TIARStep2},
i.e., the complexity of Step~2 in the plain TIAR. 
Notice that when $n$ is sufficiently large the dominating term of the complexity 
of WTIAR is $\mathcal{O}(nm^2)$ which
 is also the complexity of the Arnoldi algorithm 
for the standard eigenvalue problem. The complexity 
is verified in practice in section~\ref{sect:simulations}. 
%In this case, solving the waveguide eigenvalue problem with our approach has the same 
%complexity than solve a linear eigenvalue problem by means of the standard Arnoldi algorithm. 
%

%\clearpage\null\newpage

\section{Numerical experiments}\label{sect:simulations}~\\

\subsection{Benchmark example}\label{sect:refwaveguide}

% SCHEME OF THE SUBSECTION:
% 
% 1-How to solve the problem (general approach) with the adaption of TIAR
% 2-Regularity of the solutions and quadratic convergence
% 3-Comparison IAR, TIAR and adaption of TIAR

% HOW TO SOLVE THE PROBLEM

In order to illustrate properties of our approach,
we consider a waveguide previously analyzed in \cite{Tausch:2000:WAVEGUIDE,Butler:1992:BOUNDARY}.
We set the wavenumber as in Figure~\ref{fig:k_tausch}, where 
$K_1=\sqrt{2.3} \ \omega$, $K_2=\sqrt{3} \ \omega$, $K_3 = \ \omega$
and %. For this test we choose 
$\omega=\pi$.
Recall that the task is to compute  the eigenvalues  
in the region  
$\Omega := (-\infty,  0) \times (-2 \pi, 0 ) \subset \mathbb{C}$, 
in particular those which are close to the imaginary axis.

\begin{figure}[h]
  \begin{center}
    \includegraphics{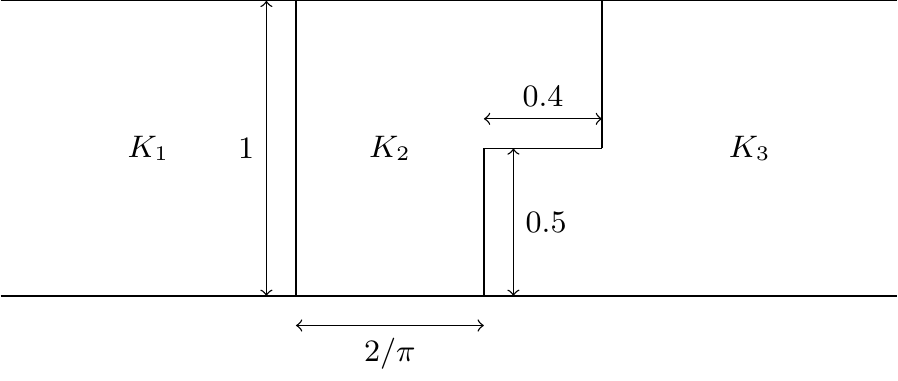}
    \caption{Geometry of the waveguide in Section~\ref{sect:refwaveguide}}\label{fig:k_tausch}
    \end{center}
\end{figure}

We select $x_-$ and $x_+$  
such that the interior domain is minimized, i.e.,
$x_-=0$ and $x_+=2/\pi+0.4$. 
The PDE is discretized with 
a FEM approach 
as explained 
in section~\ref{sect:intdisc}.
Recall that the waveguide eigenvalue problem 
has branch point singularities and 
that the algorithms we are considering 
are based on Taylor series expansion. 
As explained in section~\ref{sect:cayley},
 the location of the
shift $\gamma$ in the Cayley transformation influences
the convergence of the Taylor series, and cannot be chosen 
too close to the target, i.e., the imaginary axis.
%See Section~\ref{sect:cayley}.
We select $\gamma_0=-3+i \pi$, i.e., 
in the middle of  $\Omega$ in the imaginary direction. 
The error is measured  using the relative residual norm
\begin{equation}\label{eq:resnorm} \small
  E(w,\gamma):=
	      \frac{\|M(\gamma)w\|}{\sum_{i=0}^2|\gamma|^i \left( \|A_i\|+\|C_{1,i}\| \right) +
\|C_2^T\|+
2d_0+\sum_{j=-p}^p \left( |s_{j,+}(\gamma)|+|s_{j,-}(\gamma)| \right)},
\end{equation}
for $\|w\|=1$. 

We discretize the problem  
and we compute the eigenvalues of the nonlinear eigenvalue problem 
using WTIAR\footnote{All simulations were carried out
with Intel octa core i7-3770 CPU 3.40GHz and 16 GB RAM, except for
the last two rows of Table~\ref{tbl:converge_Tausch}
which were computed  with Intel Xeon 2.0 GHz and 64 GB RAM.
}. They are reported, for different discretizations,
in Table~\ref{tbl:converge_Tausch}. The
required CPU time is reported in Table~\ref{tbl:cputime_and_memory_IARvsTIAR}. The solution 
of the problem with the
finest discretization, i.e., the last row in 
\ref{tbl:converge_Tausch}, was computed in more than 10 hours. The bottleneck for the finest discretization is 
the memory requirements of the
 computation of the LU-factorization of the Schur complement
corresponding to $\tilde{M}(0)$. 
 %, see section~\ref{sect:tiar_adapt}.
%It turns out that there are two eigenvalues close to the imaginary axis. 
An illustration of the execution of WTIAR, for this problem, is given in 
Figure~\ref{fig:simulation_Tausch}, where 
the domain is discretized with $n_x=640$ and $n_z=641$ 
and $m=100$ iterations are performed. 
We observe in Figure~\ref{sfig:err_hist_tausch}
 that two Ritz values converge within the region of interest 
and two additional approximations converge 
to values with positive real part and of no interest in this application.
%see Figure~\ref{sfig:err_hist_tausch}.

\begin{figure}
\ffigbox[]{
\begin{subfloatrow}
  \ffigbox[\FBwidth]
    {\caption{Relative residual norm history}\label{sfig:err_hist_tausch}}
    {\scalebox{0.7}{\includegraphics{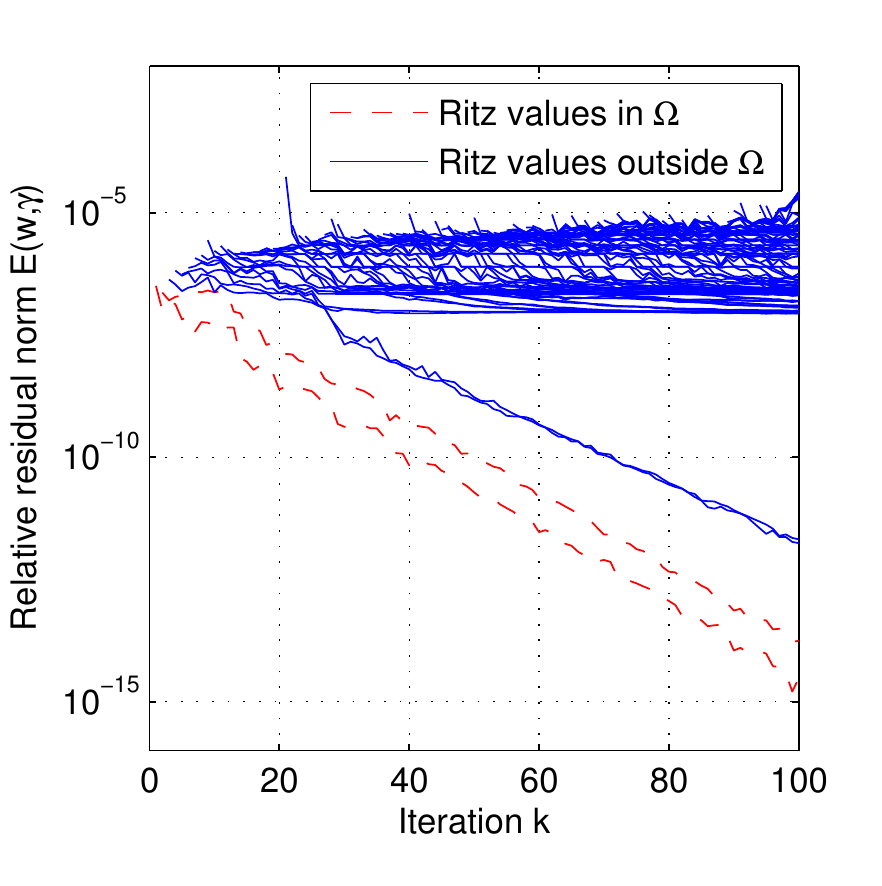}}}
\end{subfloatrow}  \vspace*{\columnsep}
\begin{subfloatrow}
\vbox{
  \ffigbox[\FBwidth]
    {\caption{Absolute value of the eigenfunction}\label{sfig:eigenfunction_tausch}}
    {\scalebox{0.7}{\includegraphics{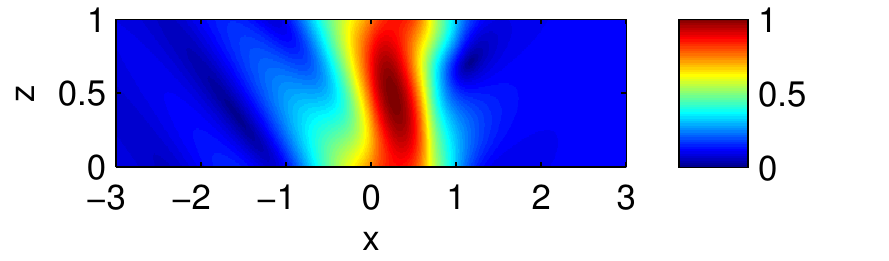}}\vspace{-0.3cm}}	\par
  \ffigbox[\FBwidth]
    {\caption{Fourier coefficients decay of $u(x_-,z)$ and $u(x_+,z)$}\label{sfig:fourier_tausch}}
    {\scalebox{0.7}{\includegraphics{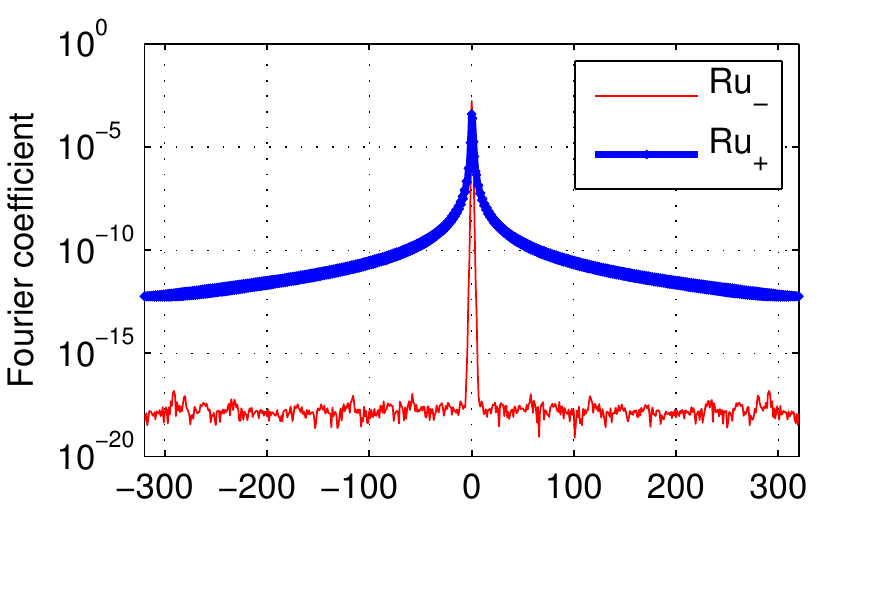}}\vspace{-0.5cm}}	
}
\end{subfloatrow}  
}{\caption{Execution of $m=100$ iterations of WTIAR. 
The parameters of the DtN maps are 
$x_-=0$ and $x_+=2/\pi+0.4$. 
The domain is discretized setting $n_x=640$ and $n_z=641$. }\label{fig:simulation_Tausch}}
\end{figure}

In Figure~\ref{sfig:fourier_tausch}
we observe that the Fourier coefficients do not
have exponential decay for $u(x_+,z)$.
Indeed, the decay is quadratic, which is 
consistent with the fact that the solutions
are $C^1$-functions, but in general not $C^2$, as
explained in Remark~\ref{remark:smooth}.
In particular,
the second derivative
of the eigenfunction is not continuous in $x=x_+$.
%We also observe that the second derivative in the direction $z$ 
%and there is a quadratic 
%decay of the Fourier coefficients
%according to the fact that the eigenfunctions are $C^1$, 
%see figure~\ref{sfig:fourier_tausch}. 
Hence, the eigenfunctions appear to have just weak regularity,
which means that the waveguide eigenvalue problem
does not have a strong solution. 
This supports the choice of the discretization method, 
based on the FEM, that we use in this paper.
In these simulations we selected $x_\pm$ such that
the enterior domain is minimized. We also carried out 
simulations for larger interior domains, without observing any
qualitive difference in the computed solutions. By Remark~\ref{remark:smooth},
this suggests that the $C^1$-assumption is not
a restriction in this case.

%Indeed, numerically, the second derivatives are discontinue 
%in the same regions where the wavenumber is discontinuous. 
%Probably our first approach, based on 
%finite difference discretization, was not efficient in terms of convergence because of this behavior of 
%the eigenfunctions.

The plot of the absolute value of one eigenfunction 
is given in Figure~\ref{sfig:eigenfunction_tausch}.
%the absolute value of the other eigenfunction 
% the same.
The convergence rate with respect to discretization, 
appears to be quadratic 
%respect--->in 
in the diameter of the 
elements. See Table~\ref{tbl:converge_Tausch}. 
% OLD WAY TO PUT PICTURES
% \begin{figure}[h]
%   \begin{center}
%    \scalebox{0.6}{\includegraphics{gfx/simulations_FEM/Err_hist.pdf}}	
%     \caption{Relative residual norm for $n_x=640$ and $n_z=641$}	\label{fig:err_hist_tausch}
%     \end{center}
% \end{figure}
% 
% In Figure~\ref{fig:eigenfunction_tausch} there is 
% the plot of the absolute value of one eigenfunction, the other eigenfunction looks the same.
% \begin{figure}[h]
%   \begin{center}
%    \scalebox{1.5}{\includegraphics{gfx/simulations_FEM/eigenfunction_Tausch.pdf}}	
%     \caption{Absolute value of the eigenfunction for $n_x=640$ and $n_z=641$ }	\label{fig:eigenfunction_tausch}
%     \end{center}
% \end{figure}
%
%
% OLD WAY TO PUT PICTURES
% \begin{figure}[h]
%   \begin{center}
%    \scalebox{1.5}{\includegraphics{gfx/simulations_FEM/fourier_decay.pdf}}	
%     \caption{Decay Fourier coefficients for $n_x=640$ and $n_z=641$ }	\label{fig:fourier_tausch}
%     \end{center}
% \end{figure}
%
% OLD STUFF
% NOW WE ILLUSTRATE THE QUADRATIC CONVERGENCE
% In order to measure the convergence rate with respect to discretization, 
% we computed the two eigenvalues discretizing the problem with different grid-size 
% . We computed the solutions with 100 steps of the adaption of TIAR.
% It turns out that the converge rate is quadratic respect the diameter of the 
% elements, see table~\ref{tbl:converge_Tausch}. \par

% 10 DIGITS USED FOR EVERY NUMBER. IT IS ENOUGH TO CLAIM QUADRATIC CONVERGENCE
\begin{table}
 \begin{center}
  \resizebox{12cm}{!} {
    \begin{tabular}{|c|c|c|c|c|}
\hline
Problem size	&	$n_x$	&	$n_z$	&	First eigenvalue		&	Second eigenvalue		\\
\hline
132		&	10	&	11	& 	-0.010297987 - 4.966269257i 	&	-0.008202089 - 1.390972357i	\\
\hline
462		&	20	&	21	& 	-0.009556975 - 4.965939619i	&	-0.009012367 - 1.337899343i	\\
\hline
1,722		&	40	&	41	& 	-0.009401369 - 4.965933116i	&	-0.009258151 - 1.322687924i	\\
\hline
6,642		&	80	&	81	& 	-0.009368285 - 4.966067569i	&	-0.009332752 - 1.318511833i	\\
\hline
26,082		&	160	&	161	& 	-0.009359775 - 4.966072322i	&	-0.009350769 - 1.317465909i	\\
\hline
103,362		&	320	&	321	& 	-0.009357649 - 4.966071811i	&	-0.009355348 - 1.317202268i	\\
\hline
411,522		&	640	&	641	& 	-0.009357159 - 4.966073495i	&	-0.009356561 - 1.317134070i	\\
\hline
1,642,242	&	1,280	&	1,281	& 	-0.009357028 - 4.966073418i	&	-0.009356859 - 1.317117443i	\\
\hline
6,561,282	&	2,560	&	2,561	&	-0.009356994 - 4.966073409i	&	-0.009356933 - 1.317113346i	\\
\hline
9,009,002	&	3,000	&	3,001	&	-0.009356991 - 4.966073406i	&	-0.009356938 - 1.317112905i	\\
\hline
\end{tabular}
}
 \end{center}
 \caption{Eigenvalue approximations stemming from WTIAR
with  $m=100$.}
 \label{tbl:converge_Tausch}
\end{table}

% I FOUND NO SPACE/REASON TO PUT THE v-function, where v=exp(gamma x) u. Even if I will put it I have no idea why and in which contest 
% refer to that. Maybe I can add it later.

% HERE WE TEST THE ALGORITHMS 
%We consider three algorithms to solve the waveguide eigenproblem: 
% We will now illustrate the difference
% between the algorithms
% IAR, TIAR and WTIAR. 
As we mentioned in Remark~\ref{rem:complexity}, 
TIAR requires less memory and has the same complexity of IAR, 
although it is in practice considerable faster. 
According to section~\ref{sect:iar_adaption}, 
WTIAR requires the same memory resources as TIAR, but WTIAR
 has lower complexity. These properties are illustrated in   
Figure~\ref{sfig:time_execution_TIARvsIAR} and 
Table~\ref{tbl:cputime_and_memory_IARvsTIAR}.
\begin{figure}
\ffigbox[]{
\begin{subfloatrow}
  \ffigbox[\FBwidth]
    {\caption{CPU time needed to perform $m=100$ iterations of IAR, TIAR and WTIAR}\label{sfig:time_execution_TIARvsIAR}}
    {\scalebox{0.7}{\includegraphics{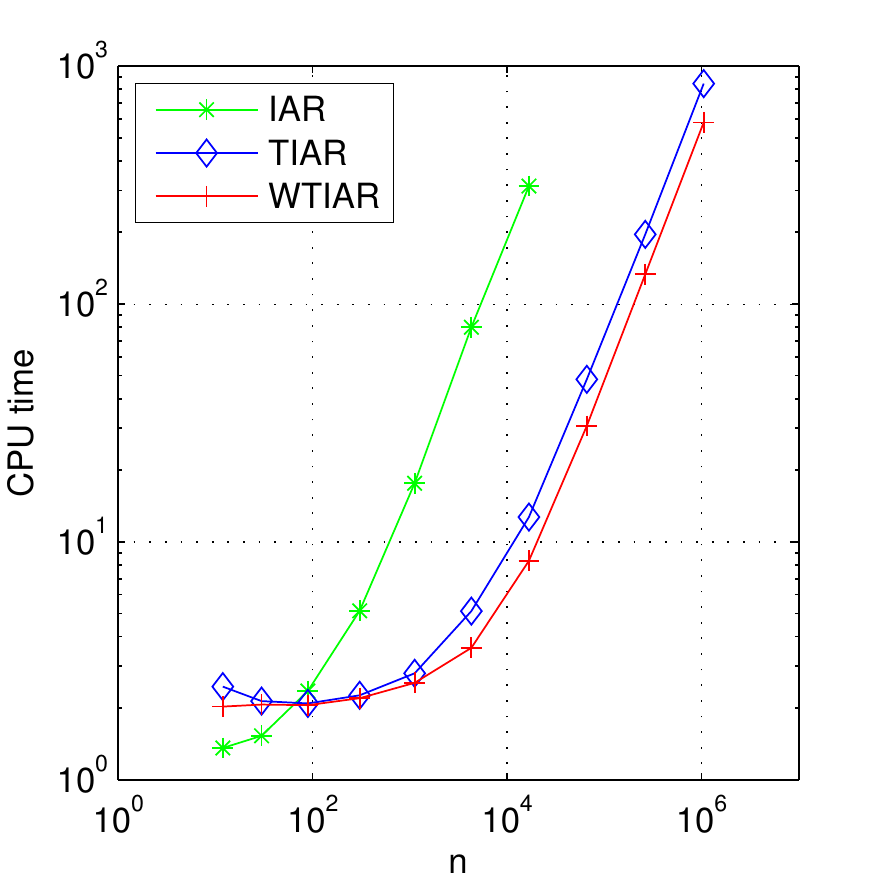}}}	\par
  \ffigbox[\FBwidth]
    {\caption{Difference between the matrices $H_{m}$ and Ritz values in $\Omega$ computed with IAR and WTIAR}\label{sfig:stability}}
    {\scalebox{0.7}{\includegraphics{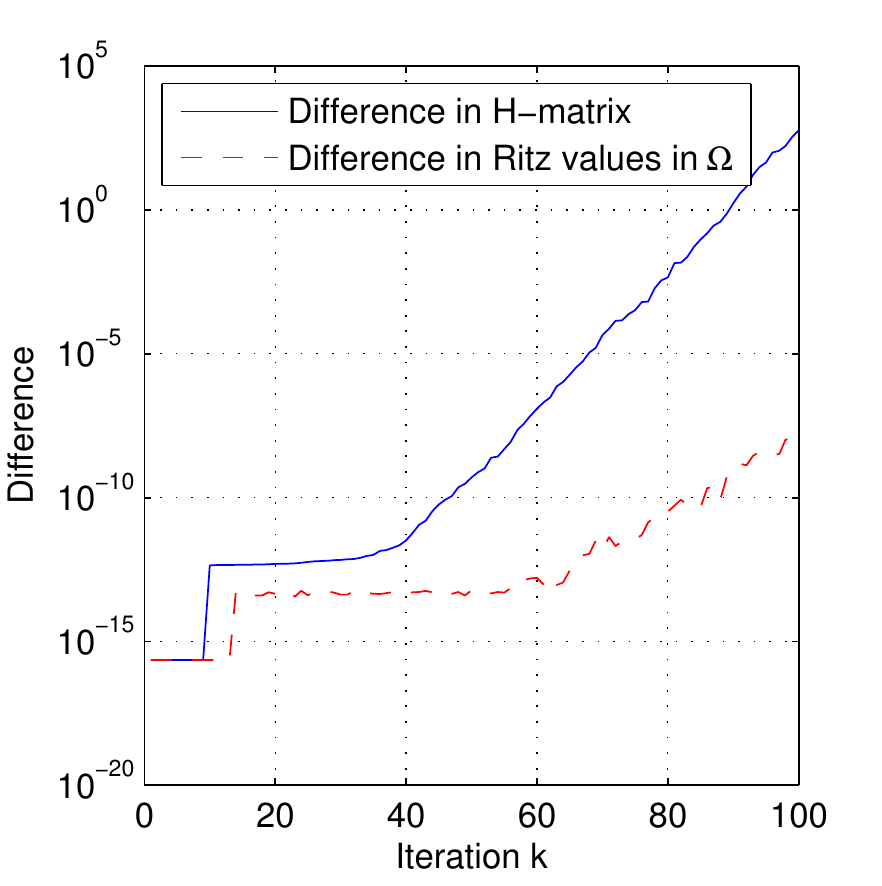}}}	
\end{subfloatrow} 
}{\caption{Comparison between IAR, TIAR and the WTIAR in terms of 
CPU time and stability}\label{fig:comparison_IAR_TIAR}}
\end{figure}
\begin{table}
 \begin{center}
   \resizebox{12cm}{!} {
    \begin{tabular}{c|c|c|c|c|c|c|}
     \cline{4-7}
		\multicolumn{3}{c}{}		&	 \multicolumn{2}{|c|}{CPU time } 			&	\multicolumn{2}{|c|}{storage of $Q_m$} 		\\
\hline
\multicolumn{1}{|c|}{$n$}		&	$n_x$	&	$n_z$	&	  IAR				&				WTIAR&	 		IAR 		&	 TIAR				\\
\hline
\multicolumn{1}{|c|}{462}		&	20	&	21	& 	 8.35 secs			&		2.58 secs		&		35.24 MB	&	7.98 MB				\\
\hline
\multicolumn{1}{|c|}{1,722}		&	40	&	41	& 	 28.90 secs			&		2.83 secs		&		131.38 MB	&	8.94 MB				\\
\hline
\multicolumn{1}{|c|}{6,642}		&	80	&	81	& 	1 min and 59 secs		&		4.81 secs		&		506.74 MB	&	12.70 MB			\\
\hline
\multicolumn{1}{|c|}{26,082}		&	160	&	161	& 	8 mins and 13.37 secs		&		13.9 secs		&		1.94 GB		&	27.52 MB			\\
\hline
\multicolumn{1}{|c|}{103,362}		&	320	&	321	& 	out of memory			&		45.50 secs		&		out of memory	&	86.48 MB			\\
\hline
\multicolumn{1}{|c|}{411,522}		&	640	&	641	& 	out of memory			&		3 mins and 30.29 secs	&		out of memory	&	321.60 MB			\\
\hline
\multicolumn{1}{|c|}{1,642,242}		&	1280	&	1281	& 	out of memory			&		15 mins and 20.61 secs	&		out of memory	&	1.23 GB				\\
\hline
\end{tabular}
}
 \end{center}
 \caption{CPU time and estimated memory required to perform $m=100$ iterations of IAR and WTIAR. The memory requirements
for the storage of the basis is the same TIAR and WTIAR.}
 \label{tbl:cputime_and_memory_IARvsTIAR}
\end{table}
As we showed in the theorem~\ref{thm:equiv} TIAR and IAR are mathematically 
equivalent by construction. 
%One can wonder if TIAR is stable and if both algorithms 
%give in practice the same solutions. 
However, IAR and TIAR (as well as WTIAR) 
incorporate orthogonalization in different
ways which may influence the impact of round-off errors.
It turns out that the $H_m$ matrices 
computed with IAR and TIAR 
are numerically different, but the Ritz values in $\Omega$
 have a small difference. See 
Figure~\ref{sfig:stability}.
This suggests that there is an effect of the roundoff errors, but for the purpose of 
computing the Ritz values located in $\Omega$, such error is benign.
Moreover, since TIAR requires less operations in the orthogonalization and implicitly imposes the structure, 
the roundoff errors might even be smaller for TIAR than IAR.

% OLD STUFF, TO REMOVE
% PUT IN THE SECTION OF THE COMPARISON BETWEEN TIAR AND IAR, IS RENDONDANT
% There are several reasons for this huge gap.
% \begin{itemize}
%   \item  	One of dominating parts of TIAR is the computation of $y_2, \dots, y_{k+1}$, while for IAR is the (double) orthogonalization. 
% 		Counting the 
% 		number of operations it turns out that the bottle neck of TIAR require 4 times less operations then the bottle neck of IAR. 
% 		That is in practice TIAR is 4 times faster than IAR.	  
%   \item		TIAR is working with less data, so the operations are faster and it is possible to have cache optimizations.  
%   \item		TIAR use matrix multiplication to compute the vectors $y_2, \dots, y_{k+1}$ that in practice reduce a lot the computation time. In fact 
% 		the matrix multiplication is very well optimized in matlab. 		
%  \item		We exploited the structure of the problem. Not all the elements of the vectors $y_2, \dots, y_{k+1}$ are needed 
% 		then we can reduce the complexity, see remark \ref{rem:complexity}.
% \end{itemize}                                
                                          
We mentioned in section~\ref{sect:iar_adaption}, 
that when $n$ became sufficiently large, 
the dominating part of WTIAR 
is the orthogonalization process.
This can be observed in Figure~\ref{fig:complexity_practice}.
Recall that the orthogonalization in WTIAR 
has complexity $\mathcal{O}(nm^2)$,
which is also the complexity 
of the standard Arnoldi algorithm. 
%for $m$ steps, see remark~\ref{rem:complexity}.
%
%
\begin{figure}
%\ffigbox[]{
\ffigbox[]{
\begin{subfloatrow}
  \ffigbox[\FBwidth]
    {\caption{IAR}\label{ab}}
    {\scalebox{0.7}{\includegraphics{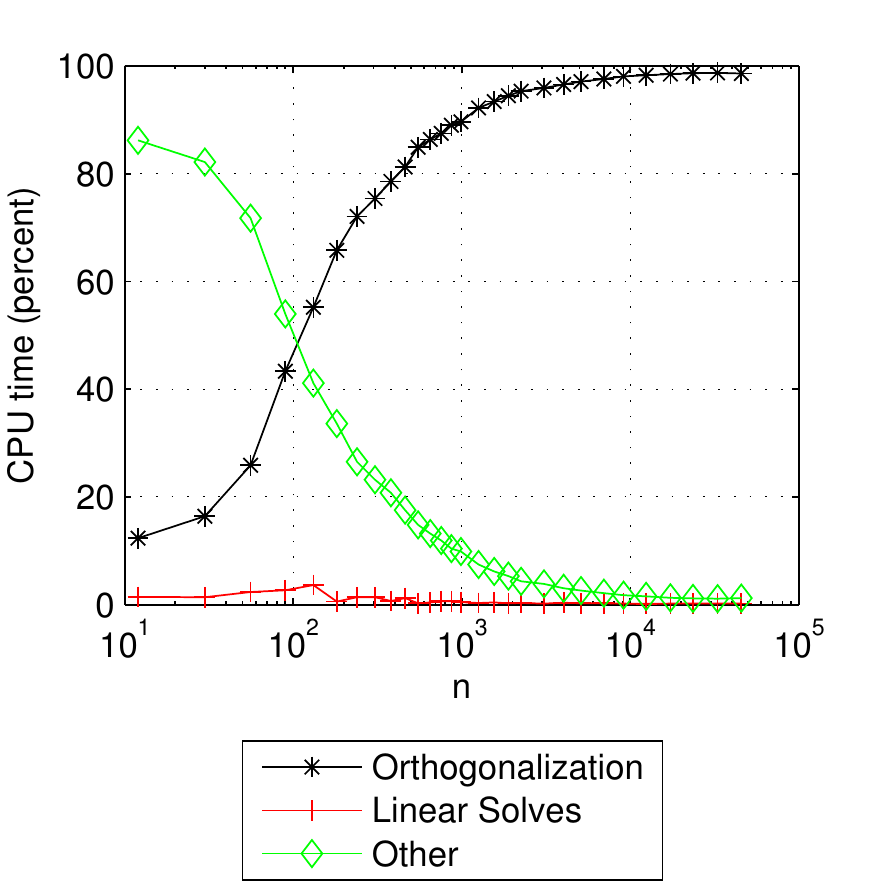}}\vspace{0.4cm}}
  \ffigbox[\FBwidth]
    {\caption{WTIAR}\label{a}}
    {\scalebox{0.7}{\includegraphics{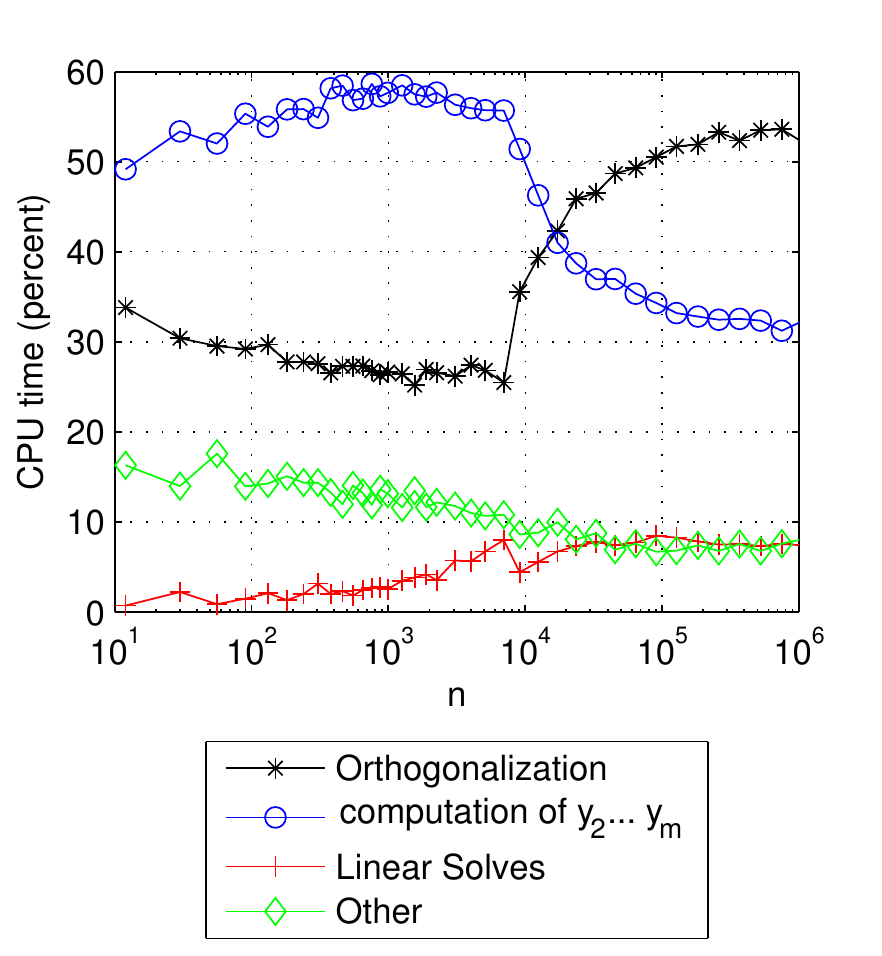}}}
\end{subfloatrow}}{\caption{CPU time (percent) of the main parts of IAR and WTIAR with $m=100$ iterations.}	\label{fig:complexity_practice}}
% %\par %\vspace*{\columnsep}
%}{\caption{Execution of TIAR. The parameters of the Dirichlet to Neumann
%map are $x_- = -1$ and $x_+ = 1$. The domain is discretized setting $n_x = 640$
%and $n_z = 641$.
%}}
%\begin{figure}[h]
%  \begin{center}
%%   \scalebox{0.7}{\includegraphics{gfx/simulations_FEM/IAR_TIAR_COMPLEXITY.pdf}}	
%   \scalebox{0.7}{\includegraphics{gfx/simulations/IAR_complexity.pdf}}	
%   \scalebox{0.7}{\includegraphics{gfx/simulations/TIAR_complexity.pdf}}	
%    \caption{ Complexity in practice (will be regenerated) }	\label{fig:complexity_practice}
%    \end{center}
\end{figure}
Hence, solving
the waveguide eigenvalue problem with WTIAR 
 using a fine discretization,
has in this sense the same complexity as solving a standard
eigenvalue problem of the same size using the Arnoldi algorithm. 
According to remark~\ref{rem:complexity} 
the dominating part of IAR 
is also the orthogonalization process, 
but this has higher complexity $\mathcal{O}(nm^3)$.
See Figure~\ref{fig:complexity_practice}a.

% WRITE SOME CONCLUSION?

%\clearpage\null\newpage
\subsection{Waveguide with complex shape} \label{sect:wg2}
In order to show the generality of our algorithm, we 
carried out simulations on a 
waveguide with a more complex geometry and solutions.
It is described in Figure~\ref{fig:k_waveguide} 
where $K_1=\sqrt{2.3} \omega$, $K_2=2 \sqrt{3} \omega$, 
$K_3=4 \sqrt{3} \omega$ and $K_4=\omega$ and $\omega=\pi$.

\begin{figure}[h]
  \begin{center}
    \includegraphics{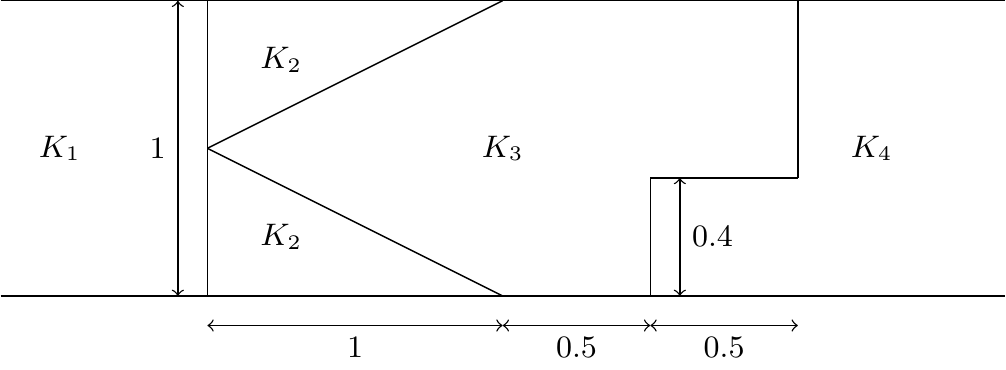} 
    \caption{Geometry of the waveguide}	\label{fig:k_waveguide}
    \end{center}
\end{figure}

We again select $x_-$ and $x_+$ 
such that the interior
domain is minimized, i.e., $x_-=0$ 
and $x_+ = 2$. We discretize the problem and 
choose the same discretization 
parameters as in section~\ref{sect:refwaveguide} and 
%we computed 
%the eigenvalues in the first 5 half--strips $S_1, \dots, S_5$
%with 
choose as 
shift 
$\gamma_0 = - 2 - i\pi$. 
An illustration of the execution of WTIAR, for this problem, is given in 
Figure~\ref{fig:simulation_challenging_waveguide}, where 
the domain is discretized with $n_x=640$ and $n_z=641$ and $m=100$ iterations are performed. 
We observe that several Ritz values converge within the region of interest 
$\Omega$. % = (-\infty,  0] \times [-2 \pi, 0 ] \subset \mathbb{C}$
See 
Figure~\ref{sfig:error_hyst_challenging} and 
Figure~\ref{sfig:eigenvalues_challenging}.
%This more challenging problem has the same characteristics as the example in section~\ref{sect:refwaveguide}.
%That is the eigenfunctions appear to have weak regularity and the 
%convergence rate of the eigenvalues of the discrete problem 
%respect to discretization 
%appear to be quadratic respect the diameter of the 
%elements.
One of the dominant eigenfunctions
is illustrated in  Figure~\ref{sfig:dominant_eigenfunction_challenge_waveguide}.

\begin{figure}
\ffigbox[]{
\begin{subfloatrow}
  \ffigbox[\FBwidth]
    {\caption{Eigenvalues (converged Ritz values) and singularities}\label{sfig:eigenvalues_challenging}}
    {\scalebox{0.7}{\includegraphics{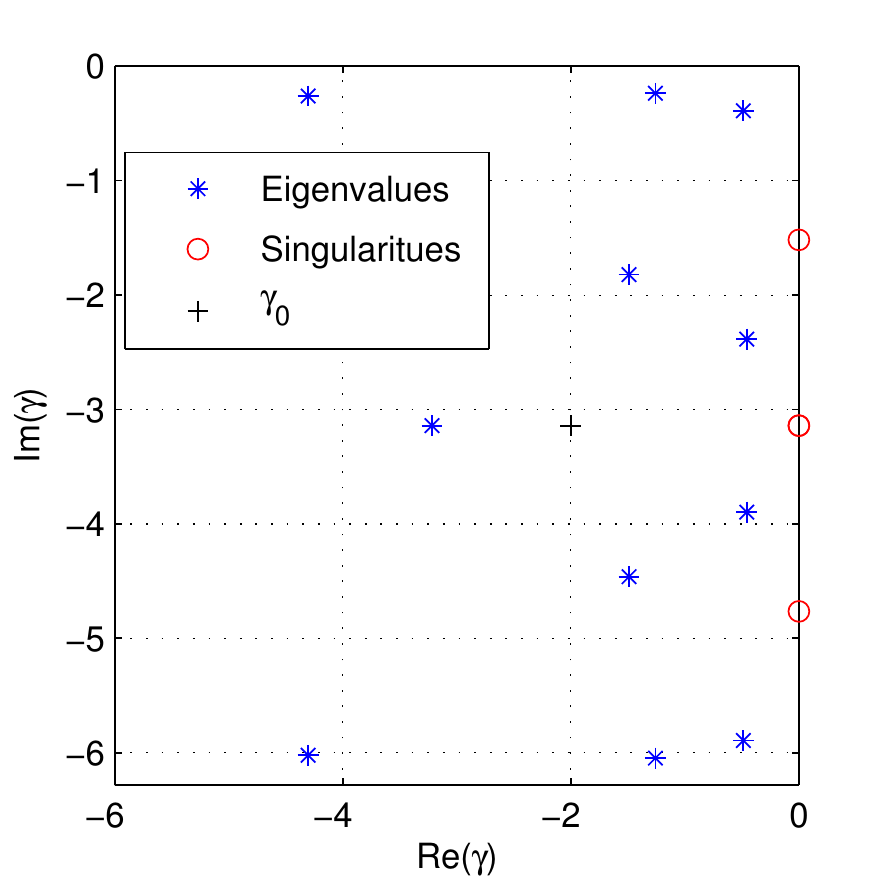}}}	
  \ffigbox[\FBwidth]
    {\caption{Relative residual norm history}\label{sfig:error_hyst_challenging}}
    {\scalebox{0.7}{\includegraphics{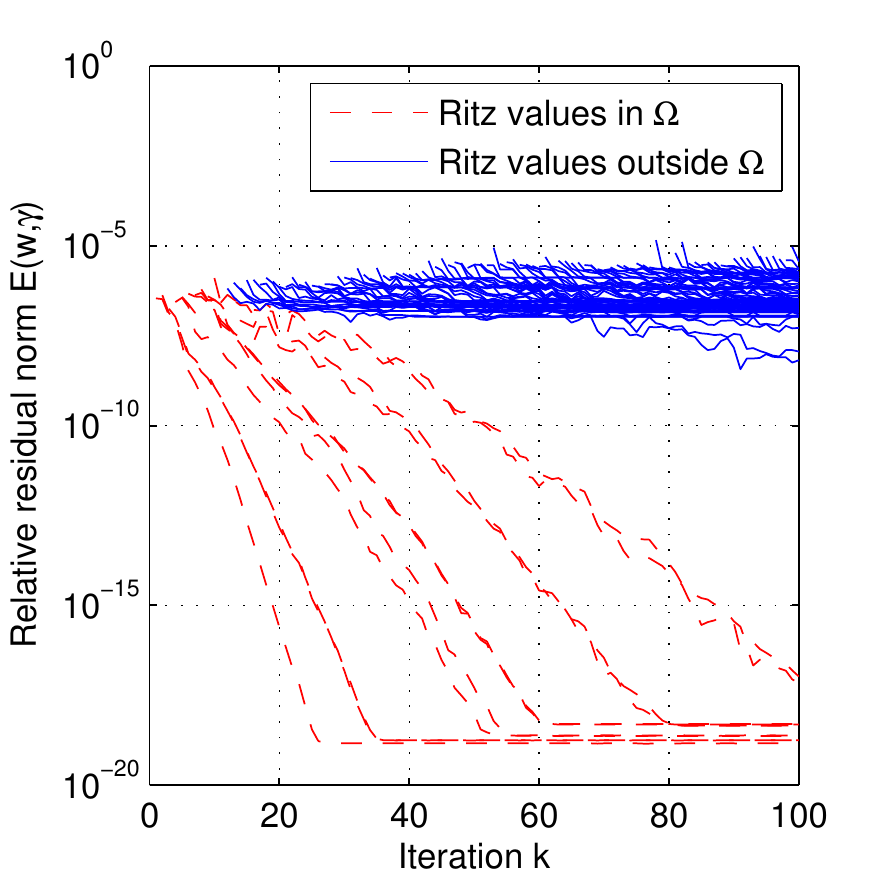}}}	
\end{subfloatrow} \par %\vspace*{\columnsep}
\begin{subfloatrow}
%  \vsize0.7\hsize
%  \vbox to 4cm{
  \ffigbox[\FBwidth]
    {\caption{Absolute value of the eigenfunction}\label{sfig:dominant_eigenfunction_challenge_waveguide}}
    {\scalebox{0.7}{\includegraphics{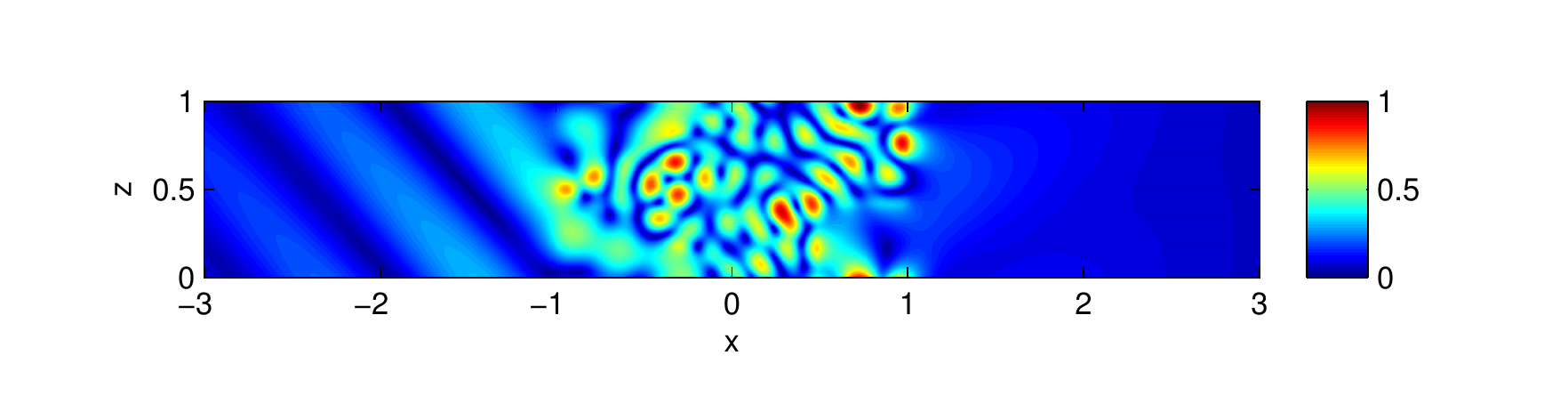}}\vspace{-0.5cm}}
%  }
\end{subfloatrow}
}{\caption{Execution of $m=100$ iterations of WTIAR. The parameters of the DtN 
maps are $x_- = 0$ and $x_+ = 2$. The domain is discretized setting $n_x = 640$
and $n_z = 641$.
}\label{fig:simulation_challenging_waveguide}}
\end{figure}

\section{Concluding remarks and outlook}

We have in this paper presented a new general approach for 
NEPs and shown how to specialize the method to 
a specific the waveguide eigenvalue problem
 stemming from analysis of wave propagation.
Note that the non-polynomial nonlinearity stems
from the absorbing boundary conditions. In our setting
we were able establish an explicit characterization
of the DtN maps, which allowed us to incorporate
the structure at an algorithmic level. This approach does not
appear to be restricted to the waveguide problem. Many
PDEs can be constructed with 
absorbing boundary conditions expressed in a closed forms.
By appropriate analysis, in particular differentiation
with respect to the eigenvalue, the approach
should carry over to other PDEs and other absorbing boundary
conditions.

There exist several variants of  IAR, e.g., the Chebyshev version 
\cite{Jarlebring:2012:INFARNOLDI} 
and restarting variations \cite{Jarlebring:2014:SCHUR}. 
There are also related rational Krylov methods
 \cite{VanBeeumen:2013:RATIONAL}. The
results of this paper may also be extendable to 
these situations, although this would require further analysis.
In particular,  all of these methods
require (in some way) a
 quantity corresponding to formula $y_1$ in \eqref{eq:y0}, 
for which the problem-dependent structure must be incorporated.
The computation of this quantity must be accurate
and efficient and require considerable problem-specific
attention. 

\appendix
\section{Proof of Lemma~\ref{thm:DtN}}\label{sect:DtNproof}
We consider the exterior problem on $S_+$ in (\ref{eq:wg_exterior}).
The proof corresponding to $S_-$ is analogous. 
To simplify the notation we
write $\beta_k=\beta_{+,k}(\gamma)$ and 
assume, without loss of
generality, that $x_+=0$.
By Remark~\ref{remark:smooth} the
solutions of \eqref{eq:wg_exterior}
are in $C^1$ and %in $H^1$ must be smooth 
%for $x>x_+$ and
every vertical trace can then be expanded in a Fourier series.
 Therefore we can express 
$$
   w(x,z) = \sum_{k\in{\mathbb Z}} w_k(x) e^{2\pi i k z},\qquad
   g(z) = \sum_{k\in{\mathbb Z}} g_ke^{2\pi i k z}.
$$
%Plugging this into the equation gives
By again using Remark~\ref{remark:smooth}, we have that 
the solutions to the exterior problem are in $C^\infty$ if $x>0$ and 
 satisfy \eqref{eq:wg_exterior1}. Therefore, 
the coefficients $w_k$ satisfy
$$
  \sum_{k\in{\mathbb Z}} \left(w_k'' - (2\pi k)^2 w_k + 4\pi \gamma i w_k + 
  (\gamma^2+\kappa_+^2)w_k
  \right)e^{2\pi i k z} = 
  \sum_{k\in{\mathbb Z}} \left(w_k'' + \beta_k w_k
  \right)e^{2\pi i k z} = 0,
$$
where $\beta_k$ is given in \eqref{eq:betadef}.
Thus, in order for $w$ to satisfy \eqref{eq:wg_exterior1}, 
we have
$$
w_k'' + \beta_kw_k=0,
$$
for all $k$. We now claim that 
there are constants $C,C'$ independent of $k$ such that
\begin{equation}\label{eq:betaclaim}
|\beta_k|\leq C(1+ (2\pi k)^2),\qquad
|\Im \sqrt{\beta_k}|\geq  C'\sqrt{1+ (2\pi k)^2}.
\end{equation}
In particular, $\beta_k\neq 0$ and
$$
   w_k(x) = c_k e^{i\sign(\im(\beta_k))\sqrt{\beta_k}x}+d_k e^{-i\sign(\im(\beta_k))\sqrt{\beta_k}x}.
$$
%where we have chosen the sign of the square root such that 
%$\Im\sqrt{\beta_k}>0$.
To determine $c_k$ and $d_k$ we have two boundary conditions.
First, 
since $w\in H^1(S_+)$, then $|w_k(x)|$ can not grow as $x\to\infty$.
This means that $d_k=0$.
Second, at $x=0$, we have $w(0,z)=g(z)$, so $c_k=g_k$.
Hence, we have the explicit solution
$$
   w_k(x) = g_k e^{i\sign(\im(\beta_k))\sqrt{\beta_k}x}.
$$
Existence is thus proved, and the relationship \eqref{eq:DtNrel} for the
DtN maps follows directly for this solution by differentiating
$w_k(x)$ and evaluating at $x=0$. We also have ${\mathcal T}_{+,\gamma}[g]\in L^2(0,1)$ since
$$
||{\mathcal T}_{+,\gamma}[g]||^2_{L^2(0,1)}
= 
\sum_{k\in{\mathbb Z}} |\beta_k|\,|g_k|^2\leq
C\sum_{k\in{\mathbb Z}} (1+(2\pi k)^2)\,|g_k|^2=C ||g||_{H^1([0,1])}.
$$
Finally, the estimate \eqref{eq:uest}
is
given by
\begin{align*}
  ||w||_{H^1(S_+)}^2 &=
  ||w||_{L^2(S_+)}^2
  +||\nabla w||_{L^2(S_+)}^2
   = \sum_{k\in{\mathbb Z}} \left[\left(1+
   (2\pi k)^2\right)
   ||w_k||_{L^2(0,\infty)}^2
   +
   ||w'_k||_{L^2(0,\infty)}^2\right]\\
  &= \sum_{k\in{\mathbb Z}} |g_k|^2\left(1 + (2\pi k)^2 + |\beta_k|
  \right)\int_0^\infty 
  e^{-2\sign(\im(\beta_k))\Im\sqrt{\beta_k}x}dx\\
  &= \sum_{k\in{\mathbb Z}} \frac{|g_k|^2}{2|\Im\sqrt{\beta_k}|}\left(1 + (2\pi k)^2 + |\beta_k|
  \right)
  \leq 
  \frac{C+1}{2C'}
  \sum_{k\in{\mathbb Z}} |g_k|^2 \sqrt{1+(2\pi k)^2}\\&=
  \frac{C+1}{2C'}||g||_{H^{1/2}([0,1])}^2.
\end{align*}
Uniqueness follows from this estimate. 
It remains to show the claim \eqref{eq:betaclaim}.
The estimate for $|\beta_k|$ is straightforward. For
the second estimate we note that $|\Im\beta_k| = 2|\Re \gamma(2\pi k+ \Im\gamma)|\neq 0$ for all $k$ 
from the assumptions $\Im\gamma\not\in 2\pi\ZZ$
and $\Re\gamma\neq 0$. It follows that also
$\Im\sqrt\beta_k\neq 0$ for all $k$ and since
$$
  \lim_{|k|\to\infty} a_k:=
  \lim_{|k|\to\infty} \frac{\Im\sqrt{\beta_k}}{\sqrt{1+(2\pi k)^2}} =
   \Im\sqrt{\lim_{|k|\to\infty} \frac{\beta_k}{1+(2\pi k)^2}} = 1,
$$
the sequence $\{1/a_k\}$ is bounded. Hence, there is a $C'$ such that
\eqref{eq:betaclaim} holds, which concludes the proof.

\section{Matrices of the FEM--discretization} \label{sect:fem}
% The matrices involved in the FEM--discretization, see section~\ref{sect:intdisc}, 
% can be expressed as
The matrices $(A_i)_{i=0}^2$ and $(C_{1,i})_{i=0}^2$ are stem from to the 
Ritz--Galerkin discretization of the bilinear forms $a, b$ and $c$. %then 
They can be decomposed and expressed as
\begin{align*}
A_0		& :=	D_{xx} + D_{zz} + K,				&	C_{1,0}		&:=	\tilde D_{xx} + \tilde D_{zz} + \tilde K,	\\
A_1		& :=	2 D_{z}, 					&	C_{1,1}		&:=	2 \tilde D_{z},					\\
A_2		& :=	B,						&	C_{1,2}		&:=	\tilde B.					
\end{align*}
Now we need to define the following
tridiagonal Toeplitz matrices. 
Let $E_m$ be the tridiagonal Toeplitz matrix
with diagonals consisting of 
$e_{i+1,i}=-1$, 
$e_{i,i}=2$ and
$e_{i,i+1}=-1$.
Let $F_m$ be the tridiagonal Toeplitz matrix
with diagonals consisting of 
$f_{i+1,i}=1$, 
$f_{i,i}=4$ and
$f_{i,i+1}=1$.
Let $G_m$ be the anti-symmetric
tridiagonal Toeplitz matrix consisting
of
$g_{i+1,i}=1$, 
$g_{i,i}=0$ and
$g_{i,i+1}=-1$.

%$E_m, F_m$ and $G_m$ of size $m \times m$ with the following Toeplitz tridiagonal 
% structure
%{\footnotesize
%\begin{align*}
% E_m=
% \begin{bmatrix}
%  2	&	-1	&		&		\\
%  -1	&	\ddots	&	\ddots	&		\\
%	&	\ddots	&		&	-1	\\
%	&		&	-1	&	2	\\
% \end{bmatrix},	\hspace{0.2cm}
%  F_m=
% \begin{bmatrix}
%  4	&	1	&		&		\\
%  1	&	\ddots	&	\ddots	&		\\
%	&	\ddots	&		&	1	\\
%	&		&	1	&	4	\\
% \end{bmatrix},	\hspace{0.2cm}
%  G_m=
% \begin{bmatrix}
%  0	&	-1	&		&		\\
%  1	&	\ddots	&	\ddots	&		\\
%	&	\ddots	&		&	-1	\\
%	&		&	1	&	0	\\
% \end{bmatrix}.
%\end{align*}
%}
Then, we have
{\footnotesize
\begin{align*}
&D_{xx} 	=	-\frac{h_z}{6h_x} 	E_{n_x} 	\otimes \left(B_{n_z}	 +e_{n_z}e_1^T+e_1e_{n_z}^T \right), 	&	&\tilde D_{xx} 	=	-\frac{h_z}{6h_x} 	-(e_1, e_{n_x})		\otimes \left(F_{n_z}	 +e_{n_z}e_1^T+e_1e_{n_z}^T \right),	\\
&D_{zz} 	=	-\frac{h_x}{6h_z} 	F_{n_x} 	\otimes \left(E_{n_z}	 -e_{n_z}e_1^T-e_1e_{n_z}^T \right),	&	&\tilde D_{zz} 	=	-\frac{h_x}{6h_z} 	 (e_1,  e_{n_x}) 	\otimes \left(E_{n_z} 	 -e_{n_z}e_1^T-e_1e_{n_z}^T \right),	\\
&D_{z} 		=	-\frac{h_x}{12} 	F_{n_x} 	\otimes \left(G_{n_z}	 +e_{n_z}e_1^T-e_1e_{n_z}^T \right),	&	&\tilde D_{z} 	=	-\frac{h_x}{12} 	 (e_1,  e_{n_x}) 	\otimes \left(G_{n_z}	 +e_{n_z}e_1^T-e_1e_{n_z}^T \right),	\\
&B		=	 \frac{h_xh_z}{36}	G_{n_x} 	\otimes \left(G_{n_z}	 +e_{n_z}e_1^T-e_1e_{n_z}^T \right),	&	&\tilde B	=	 \frac{h_xh_z}{36}	 (e_1, -e_{n_x})  	\otimes \left(G_{n_z}	 +e_{n_z}e_1^T-e_1e_{n_z}^T \right).
\end{align*}
}
\noindent The matrices $K$ and $\tilde K$ arise from the Ritz--Galerkin discretization of the bilinear form 
\[
 f(u,v) = \int_0^1 \int_{x_-}^{x_+} \kappa(x,z)^2 u(x,z) v(x,z) \ dx \ dz.
\]
The elements of such matrices are obtained integrating the product of two basis 
functions against the square of the wavenumber. We can split the 
integral over the elements. Recall that $\kappa(x,z)$ is piecewise constant, then 
the final task is to compute, for each element, the integral of a piecewise 
polynomial function. Such integral is given in an explicit form by quadrature formulas.

\section{Computation of the derivatives in DtN-map}\label{sect:alphas}
% The computation of $\alpha_{\pm,j, t}(\gamma_0)$ involve the derivatives (in zero) of the 
% square root of a polynomial 
% of second degree. 
% It is possible to compute these coefficients with a three terms recurrence.
In the computation of $y_1$ described in section~\ref{sect:iar_adaption},
we need the coefficients $\alpha_{\pm,j, \ell}$. They
can be computed with the following three-term recurrence. 
\begin{lem}[Recursion for $\alpha_{\pm,j,\ell}$]
Suppose $\gamma_0\not\in i\RR$. Then
the coefficients in \eqref{eq:coeff_DtN} 
are explicitly given by 
%It holds 
\begin{equation}
 \begin{cases}
  \alpha_{\pm,j, 0} = i w_j f_{\pm,j,0} + d_0		&			\\
  \alpha_{\pm,j, 1} = i w_j f_{\pm,j,1} - d_0		&			\\
  \alpha_{\pm,j, \ell} = i w_j f_{\pm,j,\ell} \ell!			&	\ell \geq 2,	\\
 \end{cases}
\end{equation}
where coefficients $f_{\pm,j,\ell} $ satisfy the following 
three-term recurrence
\begin{equation}\label{eq:f_recursion}
\begin{cases}
\displaystyle
f_{\pm,j,\ell} 	= 
 - \frac{2 a_{\pm,j} (\ell-3) f_{\pm,j,\ell-2} + b_{\pm,j} (2 \ell -3)  f_{\pm,j,\ell-1} }{2 \ell c_{\pm,j}}	
			&	\ell \geq 2,									\\
\displaystyle 
 f_{\pm,j,0} 	= \sqrt{ c_{\pm,j}}	,								\\
 \displaystyle 
 f_{\pm,j,1} 	= \frac{b_{\pm,j}}{2 \sqrt{ c_{\pm,j}}}	,					\\
\end{cases}
\end{equation}
with
\[
 \begin{cases}
  \displaystyle
  a_{\pm,j} 	= \conj{\gamma_0}^2 - 4 \pi i j \conj{\gamma_0} - 4 \pi^2 j^2 + \kappa_{\pm}^2	,			\\
  b_{\pm,j} 	= 2 \gamma_0 \conj{\gamma_0} + 4 \pi i j (\conj{\gamma_0} - \gamma_0) + 8 \pi^2 j^2 - 2 \kappa_{\pm}^2,	\\
  c_{\pm,j} 	= 4 \pi i j \gamma_0 - 4 \pi^2 j^2 + \kappa_{\pm}^2+\gamma_0^2	,				\\
  w_{j}		= \sign \left( \re(\gamma_0) \left( \im (\gamma_0) + 2 \pi j \right)  \right).
 \end{cases}
\]
\end{lem}

\begin{proof}
 By definition \eqref{eq:coeff_DtN}
%  \begin{align*}
% \alpha_{\pm,j,t}(\gamma_0)	&
% =\left(\frac{d^t}{d\lambda^t}
% \left((1-\lambda) \left(s_{\pm,j} \left(
% \frac{\gamma_0+\lambda\conj{\gamma}_0}{1-\lambda}\right) +d_0 \right)
% \right) \right)_{\lambda=0}						\\
% 				&
% =\left(	\frac{d^t}{d\lambda^t}
% \left((1-\lambda) s_{\pm,j} 
% \left( \frac{\gamma_0+\lambda\conj{\gamma}_0}{1-\lambda}
% \right) \right) \right)_{\lambda=0} + 
% \left( \frac{d^t}{d\lambda^t} \left( (1-\lambda) d_0 \right) \right)_{\lambda=0}
% \end{align*}
 \begin{equation*}
  \alpha_{\pm,j,\ell} =
  \left(	\frac{d^\ell}{d\lambda^\ell}
\left((1-\lambda) s_{\pm,j} 
\left( \frac{\gamma_0+\lambda\conj{\gamma}_0}{1-\lambda}
\right) \right) \right)_{\lambda=0} + 
\left( \frac{d^\ell}{d\lambda^\ell} \left( (1-\lambda) d_0 \right) \right)_{\lambda=0}.
 \end{equation*}
The computation of the second term is straightforward. The
first term can be computed as follows.
In order to compute the derivatives in zero, we now 
derive formulas for the Taylor expansion
\[
 (1-\lambda) s_{\pm,j} 
\left( \frac{\gamma_0+\lambda\conj{\gamma}_0}{1-\lambda}
\right) = 
\sum_{\ell=0}^{+\infty} f_{\pm,j,\ell} \lambda^\ell.
\]
%Once we have such Taylor expansion the proof is complete. 
%
%We now several times use reasonin
Since all functions are analytic in the origin, there exists a neighborhood of the origin $N$, such
that when $\lambda\in N$, 
%Using \eqref{eq:sk} and with a direct computation we have 
\[
  (1-\lambda)  s_{\pm,j} 
\left( \frac{\gamma_0+\lambda\conj{\gamma}_0}{1-\lambda}
\right) = w_{j}  
 \sqrt{ a_{\pm,j} \lambda^2 + b_{\pm,j} \lambda + c_{\pm,j}}.
%(1-\lambda) 
%\sqrt{ \frac{a_{\pm,j} \lambda^2 + b_{\pm,j} \lambda + c_{\pm,j}}{(1-\lambda)^2} }.
\]
%Observe that if $|\lambda|<1$ we have that
%\[
%(1-\lambda) \sqrt{ \frac{a_{\pm,j} \lambda^2 + b_{\pm,j} \lambda + c_{\pm,j}}{(1-\lambda)^2} }
%=
%h(\lambda) \sqrt{a_{\pm,j} \lambda^2 + b_{\pm,j} \lambda + c_{\pm,j} }, 
%\]
%where
%\[
%h(\lambda) = \sign \left( \re \sqrt{ \frac{a_{\pm,j} \lambda^2 + b_{\pm,j} \lambda + c_{\pm,j}}{(1-\lambda)^2} } \right) .
%\]
%It is straightforward to verify that it exists a disk containing zero 
%where $h(\lambda) \equiv 1$.
%Recall that two analytic 
%functions that are equal in disk 
%containing zero have the same Taylor expansion in zero. 
Hence, we have reduced the problem to computing the power series expansion of
 $\sqrt{a_{\pm,j} \lambda^2 + b_{\pm,j} \lambda + c_{\pm,
j}}$. 
To this end we use the well known formula involving the Gegenbauer polynomials and their generating function. 
See e.g.~\cite{polyanin2006handbook}. We have that
\begin{align*}
 \sqrt{a_{\pm,j} \lambda^2 + b_{\pm,j} \lambda + c_{\pm,j}} &= 
 \sqrt{c_{\pm,j}} \left[ \left( \sqrt{\frac{a_{\pm,j}}{c_{\pm,j}}} \lambda \right)^2 
 + \frac{b_{\pm,j}}{\sqrt{a_{\pm,j} c_{\pm,j}}} \left( \sqrt{\frac{a_{\pm,j}}{c_{\pm,j}}} \lambda \right) + 1
 \right]^{-\frac{1}{2}} \\
    & =\sum_{\ell=0}^{+\infty} \sqrt{\frac{a_{\pm,j}^\ell}{c_{\pm,j}^{\ell-1}}} C_\ell^{(-1/2)} \left( - \frac{b_{\pm,j}}{2 \sqrt{a_{\pm,j} c_{\pm,j}} } \right) \lambda^\ell,
\end{align*}
where $C_\ell$ is the $\ell$-th Gegenbauer polynomial.
%  
% 
% 
% Now we can use the well known equality involving the Gegenbauer polynomials \mytodo{cite something?}
% \[
%  (1-2Py+y^2)^{-\alpha} = \sum_{t=0}^{+\infty} C_n^{(\alpha)} (P) y^t
% \]
% and we finally find
% \[
%   \sqrt{a_{\pm,j} \lambda^2 + b_{\pm,j} \lambda + c_{\pm,j}}
%   =
%   \sum_{t=0}^{+\infty} \sqrt{\frac{a_{\pm,j}^n}{c_{\pm,j}^{n-1}}} C_n^{(-1/2)} \left( - \frac{b_{\pm,j}}{2 \sqrt{a_{\pm,j} c_{\pm,j}} } \right) \lambda^t
% \]
Consequently, the coefficients in the power series expansion are 
\[
 f_{\pm,j,\ell} 	= 
\sqrt{\frac{a_{\pm,j}^\ell}{c_{\pm,j}^{\ell-1}}} C_\ell^{(-1/2)} \left( - \frac{b_{\pm,j}}{2 \sqrt{a_{\pm,j} c_{\pm,j}} } \right).
\]
%Using the three therm recursion of the Gegenbauer polynomials we can conclude.
The recursion \eqref{eq:f_recursion} follows from substitution of the recursion formula
for Gegenbauer polynomials.
\end{proof}

\bibliographystyle{elsart-num-sort}
\bibliography{eliasbib,misc}

%% Figure example
%\begin{figure}[h]
%  \begin{center}
%    %%Traditional figure import with PDFs:
%    %\subfigure[mytext]{\scalebox{0.65}{\includegraphics{gfx/file..pdf}}}
%    %\scalebox{0.7}{\includegraphics{gfx/file.pdf}}
%    %
%    %6% Fancy-pantsy tikz figure:
%    \input{gfx/myfig1.tex}
%    %\subfigure[hello]{\input{gfx/myfig1.tex}}%
%    %\subfigure[hello]{\input{gfx/myfig1.tex}}%
%    \caption{
%      My caption
%      %\label{fig:}
%    }
%  \end{center}
%\end{figure}
%

\end{document}